\RequirePackage{fix-cm}

\documentclass[smallextended]{svjour3opt} 
\usepackage{etex}
\usepackage{todonotes}
\usepackage{palatino,url}
\usepackage{verbatim, amssymb,amsmath}
\usepackage{varwidth}
\usepackage[export]{adjustbox}
\usepackage{multirow}

\usepackage{mathrsfs}
\usepackage{graphics,amsopn,amstext,amsfonts,color}
\usepackage{bm}
\usepackage{setspace}
\usepackage{subfig}
\usepackage{graphicx}
\usepackage{epstopdf}
\usepackage{caption}
\usepackage{multirow}
\usepackage{booktabs}
\usepackage{paralist}
\usepackage{colortbl}
\usepackage{appendix}
\usepackage{soul,caption,subfig}
\usepackage[bookmarksnumbered=true,hidelinks]{hyperref}
\usepackage{graphics,latexsym,amsfonts,verbatim,amsmath}
\usepackage{algorithm}
\usepackage{algorithmicx}
\usepackage{algpseudocode}
\usepackage{tikz}
\usetikzlibrary{calc}
\usetikzlibrary{plotmarks}

\usepackage{thmtools} 
\usepackage{thm-restate}
\usepackage{pgfplots}

\usepackage{pgf}
\usepackage{pst-func}

\definecolor{pinegreen}{rgb}{0.0, 0.47, 0.44}
\usepackage{pgfpages}
\usepackage{todonotes}
\usepackage{hhline} 
\usepackage{empheq}
\usepackage{cases}

\def\Re{\mathbb{R}}

\def\hat{\widehat}

\def\Re{{\mathbb R}}

\newcommand{\exclude}[1]{}

 \newcommand{\y}{\mathbf{y}}

\renewcommand{\S}{\mathcal{S}}

\algdef{SE}[DOWHILE]{Do}{doWhile}{\algorithmicdo}[1]{\algorithmicwhile\ #1}%

\usepackage[flushleft]{threeparttable}

\def\Re{\mathbb{R}}

\def\hat{\widehat}

\def\Re{{\mathbb R}}

\def\x{\vect{x}}
\DeclareMathOperator{\tr}{tr}

\DeclareMathOperator{\argmax}{argmax}

\newcommand{\subparagraph}{}
\usepackage{titlesec}

\newtheorem{observation}{Observation}

\newcommand*{\qedA}{\hfill\ensuremath{\diamond}}

\usepackage[numbers,sort&compress,comma]{natbib}
\allowdisplaybreaks

\title{Beyond Symmetry: Best Submatrix Selection for the Sparse Truncated SVD}
\date{\today}
\titlerunning{Beyond Symmetry: Best Submatrix Selection for the Sparse Truncated SVD}

\author{Yongchun Li \and Weijun Xie}

\institute{
	First Author: Yongchun Li \at
	Affiliation: Virginia Tech, Blacksburg, VA\\
	\email{liyc@vt.edu}
	\and
	Corresponding Author: Weijun Xie \at
Affiliation: Virginia Tech, Blacksburg, VA\\
\email{wxie@vt.edu}
}

\allowdisplaybreaks

\makeatletter
\pgfmathdeclarefunction{erf}{1}{%
\begingroup
\pgfmathparse{#1 > 0 ? 1 : -1}%
\edef\sign{\pgfmathresult}%
\pgfmathparse{abs(#1)}%
\edef\x{\pgfmathresult}%
\pgfmathparse{1/(1+0.3275911*\x)}%
\edef\t{\pgfmathresult}%
\pgfmathparse{%
1 - (((((1.061405429*\t -1.453152027)*\t) + 1.421413741)*\t
-0.284496736)*\t + 0.254829592)*\t*exp(-(\x*\x))}%
\edef\y{\pgfmathresult}%
\pgfmathparse{(\sign)*\y}%
\pgfmath@smuggleone\pgfmathresult%
\endgroup
}
\makeatother

\begin{document}
\maketitle

\begin{abstract} 
Truncated singular value decomposition (SVD), also known as the best low-rank matrix approximation  with minimum error measured by a unitarily invariant norm, has been successfully applied to many domains such as biology,  healthcare, and  others, where high-dimensional datasets are prevalent. To extract interpretable information from the high-dimensional data, sparse truncated SVD (SSVD) has been used to select a handful of rows and columns of the original matrix along with the best low-rank approximation. Different from the literature on SSVD focusing on the top singular value or compromising the sparsity for the seek of computational efficiency, this paper presents a novel SSVD formulation that can select the best submatrix precisely up to a given size to maximize its truncated Ky Fan norm. The fact that the proposed SSVD problem is NP-hard motivates us to study effective algorithms with provable performance guarantees. To do so, we first reformulate SSVD as a mixed-integer semidefinite program, which can be solved exactly for small- or medium-sized instances within a  branch-and-cut algorithm framework with closed-form cuts and is extremely useful for evaluating the solution quality of approximation algorithms. We next develop three selection algorithms based on different selection criteria and two searching algorithms, greedy and local search. We prove the approximation ratios for all the approximation algorithms and show that all the ratios are tight when the number of rows or columns of the selected submatrix is no larger than half of the  data matrix, i.e., our derived approximation ratios are unimprovable. Our numerical study demonstrates the high solution quality and computational efficiency of the proposed algorithms. Finally, all our analysis can be extended to row-sparse PCA.
\end{abstract}

\newpage
\section{Introduction}
	Singular value decomposition (SVD) is a common tool for 
data analysis in statistics and computer science. Given a data matrix $\bm A \in \Re^{m\times n}$, its SVD is of the form
\begin{align*} 
\text{(SVD)} \quad	\bm A := \sum_{i\in [r]} \sigma_i\bm u_i \bm v_i^{\top},
\end{align*}
where $r$ denotes the rank of $\bm A$, and for each $i \in [r]$, $\sigma_i\in \Re_+$ is the $i$th largest singular value of $\bm A$, $\bm u_i \in \Re^{m}$ and $\bm v_i \in \Re^n$ are the corresponding left- and right-singular vectors, respectively. 
In practice, data scientists often approximate data matrix $\bm A$ by a low-rank matrix for various purposes such as information retrieval, modeling convenience, and complexity reduction (see, e.g., review papers \cite{markovsky2012low,golub1987generalization}). 
According to the generalized Eckart-Young theorem in \cite{golub1987generalization}, the best rank-$k$ approximation of matrix $\bm A$ with $k \le r$ 
that minimizes the approximation error measured by a unitarily invariant norm (e.g., Frobenius, spectral, or nuclear norm), 
is achieved by its ``$k$-truncated SVD'' (i.e., $ \sum_{i\in [k]} \sigma_i \bm u_i \bm v_i^{\top}$).
As shown in \cite{watson1993matrix}, the  $k$-truncated SVD of matrix $\bm A$ admits the following equivalent mathematical programming formulation
\begin{align} \label{tsvd}
\|\bm A\|_{(k)}	:=\max_{\begin{subarray}{c}
	\bm U\in \Re^{m\times k}, \bm V \in \Re^{n\times k}
	\end{subarray}}  \left\{\tr(\bm U^{\top} \bm A \bm V): \bm U^{\top} \bm U = \bm V^{\top} \bm V = \bm I_{k}  \right\},
\end{align}
where the Ky Fan $k$-norm $\|\cdot\|_{(k)}$ is defined as the sum of $k$ largest singular values of a matrix.
Particularly, suppose that $(\bm U^*, \bm V^*)$ is an optimal solution to problem \eqref{tsvd} with $\bm u_i^*$ and $\bm v_i^*$ denoting their $i$th columns, respectively, for each $i\in [k]$, then we must have $\bm u_i^* = \bm u_i$, $\bm v_i^* = \bm v_i$, $ (\bm u_i^*)^{\top}\bm A \bm v_i^* = \sigma_i$ for each $i\in [k]$. Thus, this optimal solution to problem \eqref{tsvd} can form the $k$-truncated SVD of matrix $\bm A$ as below
\begin{align}\label{eq_cons}
\text{(Truncated SVD) } \quad \sum_{ i \in [k]} (\bm u_i^*)^{\top}\bm A \bm v_i^*  \bm u_i^* (\bm v_i^*)^{\top}.
\end{align}

\subsection{Model Formulation of Our Sparse Truncated SVD}

Albeit being widely-used 
for data  simplification, denoising, and extraction of a large-scale matrix $\bm{A}$ (i.e., either $m$ or $n$ or both of them are large),  the $k$-truncated SVD, as the  best low-rank approximation,  can have difficulties in handling the following situations: (i) for the high-dimensional data, interpretable statistical estimations often require low-rank and sparse data structures (see, e.g., \cite{yang2014sparse});
(ii) low-rank data approximation also often comes with the extraction of rows and columns \citep{doan2016finding}. For example,
the biclustering of microarray data seeks to identify a subset of  rows (e.g., cancers) and a subset of columns (e.g., genes) of a matrix that are significantly related \citep{yang2016rate, gao2016optimal} and for the defect diagnostics in multistage manufacturing, it is desirable to achieve  a joint selection of crucial row and column variables representing stages and processes \cite{jeong2022two};
and (iii) the vanilla $k$-truncated SVD often involves many rows and columns of a matrix, 
resulting in uninterpretable factors that do not offer insights \citep{park2020multiresolution,witten2009penalized}. 
An intuitive way of resolving these issues is finding a simultaneously sparse and low-rank approximation of  matrix $\bm A$.
That is, we propose a Sparse truncated SVD (SSVD) formulation
\begin{align} \label{ssvd}
\text{(SSVD)}\quad 	w^*:=\max_{\begin{subarray}{c}
	\bm U\in \Re^{m\times k}, \bm V \in \Re^{n\times k}
	\end{subarray}}  \left\{\tr(\bm U^{\top} \bm A \bm V): \bm U^{\top} \bm U = \bm V^{\top} \bm V = \bm I_{k}, ||\bm U||_0 \le s_1, ||\bm V||_0 \le s_2  \right\},
\end{align}
where for a matrix $\bm X$, we let $||\bm X||_0$ denote the number of non-zero rows,  integers $s_1 \le m$ and $s_2\le n$ denote the largest numbers of non-zero rows of top left- and right- singular matrices, respectively, $k \le \min\{s_1, s_2\}$ denotes the desirable low rank,  and $w^*$ denotes the optimal value of SSVD. 

Given an optimal solution to SSVD \eqref{ssvd}, following the construction in \eqref{eq_cons},
we {can} obtain a sparse $k$-truncated SVD whose left- and right- singular matrices are still orthonormal and have at most $s_1$ and $s_2$ non-zero rows, respectively. 
Strictly enforcing the sparsity allows one to select at most $s_1$ rows and $s_2$ columns of the original data matrix and explore the inherent patterns as well as the low-rank approximation, and thus can provide a
neat approximation of the data matrix
and improve the interpretability of the  obtained
truncated SVD.
In addition, SSVD \eqref{ssvd} can be viewed as finding the best sparse and low-rank approximation of a data matrix. For many large-scale machine learning and statistical applications in various domains such as biology, manufacturing, healthcare, and others, where high-dimensional datasets are prevalent, it is desirable to pursue simultaneously sparse and low-rank approximation for better interpretable results. For example, by combining the low-rankness and sparsity, the quality of data recovery and reconstruction in signal/image processing can be considerably improved, compared to the methods solely adopting one of them \citep{chen2018simultaneously,niu2021seismic}. Our numerical study on solar flare detection via satellite images further validates the effectiveness of the proposed sparse and low-rank approximation in our SSVD \eqref{ssvd}.

	\noindent\textit{{ Model Interpretation.}} Combining the $k$-truncated SVD \eqref{tsvd} and the SSVD \eqref{ssvd}, we see that the SSVD \eqref{ssvd} can be recast as the following combinatorial optimization problem:
\begin{align} \label{eq_ssvdcom}
\text{(SSVD)} \quad w^*:=	\max_{\begin{subarray}{c}
	S_1\subseteq [m],   
	S_2 \subseteq [n]
	\end{subarray}
} \left\{ ||\bm A_{S_1, S_2}||_{(k)}: |S_1|\le s_1, |S_2|\le s_2\right\},
\end{align}
where for any pair of subsets $S_1 \subseteq [m]$ and $S_2 \subseteq [n]$, 
$\bm A_{S_1, S_2}$ denotes the submatrix of $\bm A$ with rows and columns indexed by $S_1$ and $S_2$, respectively.

%
The combinatorial formulation \eqref{eq_ssvdcom} leads to a more intuitive explanation of SSVD from the perspective of submatrix selection, which is remarked below:
(i) for a given data matrix $\bm A$, the objective of SSVD \eqref{eq_ssvdcom} is to select the best submatrix of size at most $s_1 \times s_2$ with the maximum Ky Fan $k$-norm; (ii) The Ky Fan $k$-norm is equivalent to the sum of all the singular values of the $k$-truncated SVD, which is commonly used to measure information contained by the top $k$ singular vectors (see, e.g., \cite{doan2016finding, xia2014optimal}); and (iii) compared to the vanilla $k$-truncated SVD, the resultant sparse rank-$k$ matrix approximation by SSVD \eqref{eq_ssvdcom} 
improves the interpretability while preserving the properties of truncated SVD (e.g., extracting $k$ largest singular values of an optimal submatrix, maintaining the orthonormality of  singular vectors, etc.).

	\noindent\textit{Model Appliance.}
The SSVD \eqref{ssvd} is versatile and generalizes many existing models. Particularly, we will show that
(i) when $s_1=m$ and $s_2=n$, it reduces to the vanilla $k$-truncated SVD \eqref{tsvd} of matrix $\bm A$; (ii) when $k=1$, it reduces to the rank-one SSVD with $L_0$ norm constraints proposed by \cite{min2015novel}; and
(iii) when matrix $\bm A$ is positive semidefinite and $s_1=s_2$, SSVD \eqref{ssvd}  reduces to the well-known row-Sparse PCA (SPCA) in \cite{vu2012minimax, dey2020solving}.  
Formally, SPCA considered in this paper can be defined as below. 
\begin{align}\label{eq_spca}
\text{(SPCA)} \quad w^{spca}:= \max_{ \bm U \in \Re^{n \times k}}  \big\{\tr(\bm U^{\top} \bm A\bm U) :  \bm U^{\top} \bm U= \bm I_k,
||\bm U||_0 \le s \big\},
\end{align}
where matrix $\bm A \in \S_+^n$ is assumed to be a sample covariance matrix and thus is positive semidefinite, $k \le s \le n$ are positive integers, and $w^{spca}$ denotes the optimal value of SPCA. Since SPCA \eqref{eq_spca} is a special case of SSVD \eqref{ssvd}, our developed  algorithms can  be adapted to all the SPCA applications.

	\subsection{Relevant Literature}
\textit{Submatrix Selection Problems.}
As the combinatorial formulation \eqref{eq_ssvdcom} interprets SSVD as a task of selecting an optimal submatrix,  we differ our SSVD \eqref{eq_ssvdcom} from some classic submatrix selection problems in literature in the following two aspects: selected submatrix form and objective function.  Given a data matrix $\bm A \in \Re^{m\times n}$, we select a rectangular submatrix of arbitrary size $s_1 \times s_2$, while the maximum volume \citep{civril2009selecting,deshpande2010efficient} and sparse ridge regression problems \citep{wei2021ideal} focus on one-dimensional subvector
selection (i.e., $s_1 =m$ or $s_2=n$); and SPCA \cite{dey2018convex} and maximum entropy sampling problems \citep{ko1995exact} seek to obtain the best principal submatrix out of a positive semidefinite matrix  $\bm A$. Moreover, different from the previous works, the optimal submatrix produced by our SSVD \eqref{eq_ssvdcom} achieves the maximum Ky Fan $k$-norm, a natural objective for modeling a low-rank matrix.

We next review the literature on CUR matrix approximation with pseudo-skeleton approximation as a special case, which aims to show that
a low-rank matrix can be well approximated using its sub-columns and sub-rows \citep{boutsidis2017optimal}. 
Formally, the CUR approximation is made up of the multiplied three matrices, $\bm C \bm U \bm R$, where $\bm C \in \Re^{m\times k}$, $\bm R \in \Re^{k \times n}$ denote the selected columns  and rows of data matrix $\bm A \in \Re^{m\times n}$, respectively, and  $\bm U \in \Re^{k\times k}$ is the kernel matrix. An obvious limitation of the  submatrix selection indicated by the CUR approximation is that the sparsity and low-rankness are not separable, i.e., it fails to guarantee a rank-$k$ approximation when selecting $s_1\ge k$ rows in $\bm C$ and $s_2 \ge k$ columns in $\bm R$, which is quite different from ours.
In fact, a majority of existing research about CUR approximation focused on an  ill-conditioned rank-$k$ matrix $\bm A$ and attempted to select a size $k\times k$ submatrix  \citep{cortinovis2020low,  tyrtyshnikov1995pseudo,zamarashkin1997pseudo} or larger-sized submatrix  \citep{osinsky2018pseudo, mikhalev2018rectangular}. In contrast, our SSVD \eqref{algo:spca_greedy} can be applied to any data matrix and selecting an arbitrary number of rows and columns to achieve a low  rank-$k$ approximation.

Other submatrix selection works have compromised to use easy-to-solve but much less accurate models by relaxing the sparsity constraints by convex ones (e.g., $\ell_1$ norm in \cite{doan2016finding}). Quite differently, our proposed SSVD  \eqref{ssvd} aims to select a prespecified-sized submatrix. As far as we are concerned, our SSVD  \eqref{ssvd} is the first-known rectangular submatrix selection model, maximizing the Ky Fan $k$-norm. Therefore,
the existing submatrix selection algorithms cannot be directly used to solve SSVD \eqref{ssvd}, and none of them admits any theoretical performance guarantees compared to the true optimal value. In the numerical study,  we tailor two commonly-used methods- randomized leverage-score algorithm \citep{mahoney2009cur} and  maximum volume-based algorithm \citep{mikhalev2018rectangular}, to solve SSVD \eqref{ssvd}  and compare them with our proposed ones.

%
%
\textit{SSVD Models and Methods.}
The concept of SSVD dates back to  the pioneering work by \cite{witten2009penalized}, which penalized the non-zero elements of rank-one truncated SVD for improving interpretability.
To better convey the advantage of the proposed SSVD formulation, we next survey the existing formulations on SSVD and explain their limitations compared to ours. 
First, many works have mainly focused on rank-one SSVD  (i.e., $k=1$ in SSVD \eqref{ssvd}) and the  $L_1$ norm constraints or regularization on the top left- and right- singular vector variables $\bm U \in \Re^{m\times 1}$ and $\bm V  \in \Re^{n\times 1}$ to obtain approximately sparse top singular vectors (see \cite{lee2010biclustering,omanovic2018knowledge,min2018group, sill2011robust,witten2009penalized}).
Albeit being computationally efficient, using $L_1$-norm may 	not always
guarantee the numbers of non-zero rows and columns in the sparse truncated SVD to be no larger than given thresholds and thus may fail to deliver desirable interpretable results.
A few works recently developed on rank-one SSVD 
impose the $L_0$ norm constraints on the top  left- and right- singular vector variables $\bm U \in \Re^{m\times 1}$ and $\bm V  \in \Re^{n\times 1}$, restricting the size of selected submatrix exactly to be less than or equal to a given threshold \citep{min2015novel, min2016l0,li2020exact}.

However, the work on the general rank-$k$ SSVD based on $L_0$ norm constraints is scarce.
In \cite{lee2010biclustering,sill2011robust,witten2009penalized}, to construct a sparse and rank-$k$ matrix approximation, they suggested selecting $k$ different $L_1$ norm constrained rank-one SSVD sequentially; namely, at each iteration, the authors proposed to subtract the sum of the obtained sparse rank-one matrices from the original data matrix and apply their rank-one SSVD methods to the resultant matrix. 
Unfortunately, their selected left- and right- singular vectors do not guarantee the orthogonality nor share the same row sparsity.
Another relevant work is \cite{yang2014sparse}, which proposed  a sparse SVD algorithm to generate $k$ sparse and orthonormal singular vectors for a better statistical estimation; however, it failed to enforce the sparsity of the obtained rank-$k$ matrix strictly. 
Different from these works, this paper studies the general rank-$k$ SSVD \eqref{ssvd} model by strictly enforcing the sparsity and orthonormality of the obtained left- and right- singular vectors simultaneously. 

\textit{Exact and Approximation Algorithms of SSVD.} 
We show that SSVD \eqref{ssvd} is NP-hard with a reduction to the rank-one SPCA,  i.e., $k=1$ in formulation \eqref{eq_spca} that has been notoriously known to be NP-hard and inapproximable (see, e.g., \cite{magdon2017np}), which motivates us to develop efficient exact and approximation algorithms. 
Notably, the rank-one SPCA has been extensively studied in the literature (e.g., see \cite{berk2019certifiably,kim2021convexification,li2020exact}).
 In \cite{li2020exact}, the authors also derived exact MISDP formulations for both rank-one SPCA and rank-one SSVD. 
As far as we are concerned, all the existing exact algorithms and formulations are not directly applicable to SSVD \eqref{ssvd}. To fill the gap, in this paper, we derive the first exact MISDP formulation and the branch-and-cut algorithm for  general rank-$k$ SSVD \eqref{ssvd}.

For the large-scale instances, it is, in general, difficult to solve SSVD \eqref{eq_misdp} or its special case SPCA to optimality within a reasonable amount of time (e.g., within an hour) even when the rank $k=1$ (see, e.g., \cite{berk2019certifiably}). To speedup the computation, another line of research on high-dimensional rank-one SPCA and rank-one SSVD has adopted effective iterative methods or approximation algorithms to find a good-quality feasible solution \citep{chan2016approximability, chowdhury2020approximation,li2020exact, min2015novel, min2016l0}. Unfortunately,  to the best of our knowledge, 
there are not any known approximation algorithms that can directly solve our proposed general rank-$k$ SSVD \eqref{eq_ssvdcom} for a possibly nonsymmetric matrix. Although some algorithms manage to yield feasible solutions for SSVD \eqref{eq_ssvdcom}, they fail to guarantee a provable performance.
Recently, the work \cite{dey2020solving} proposed a modified local search algorithm for the general rank-$k$ SPCA but the authors did not show the approximation ratio and their algorithm cannot be extended to nonsymmetric rank-$k$ SSVD.
This paper proposes new selection algorithms and successfully tailors greedy and local search algorithms to solve SSVD \eqref{ssvd} with provable approximation ratios.
Our first selection algorithm matches the best-known approximation ratio (i.e., $1/\sqrt{\min\{s_1, s_2\}}$) for rank-one SSVD in \cite{li2020exact} and for the rank-$k$ SPCA, as a special case of SSVD \eqref{ssvd}, all the proposed approximation algorithms come with the same or better performance guarantees.
	\subsection{Summary of Main Contributions}
To solve SSVD \eqref{ssvd}, we derive an equivalent MISDP formulation based on the semidefinite representation of Ky Fan $k$-norm and effective approximation algorithms using various selection and searching criteria. 
Below we list the major contributions.

\begin{enumerate}[(i)]
	\item Based on the MISDP,  we derive  a family of valid inequalities and a branch-and-cut algorithm for SSVD \eqref{ssvd},  which can solve the small- and medium-sized instances to optimality (e.g., $m=n=500$, $s_1=s_2=5$, $k=2$) and help evaluate the solution quality of our proposed approximation algorithms;
	
	\item Inspired by the combinatorial formulation  \eqref{eq_ssvdcom} of SSVD,  we consider three different selection criteria that are related to its objective function (i.e., the Ky Fan $k$-norm) yet much easier to compute,  
	and then propose three selection algorithms;
	
	\item 
	We successfully customize the well-known greedy and local search algorithms to solve SSVD \eqref{eq_ssvdcom};
	\item The approximation ratios  and time complexities of our proposed approximation algorithms for SSVD \eqref{eq_ssvdcom} are displayed in Table~\ref{table:ratio}. We prove that all the ratios are tight, i.e., un-improvable {when the sparse parameters satisfy $s_1\le m/2$ and $s_2\le n/2$};
	\item We remark that the approximation ratio of selection Algorithm~\ref{algo:trun1} is independent of $k$ and matches the best-known  one of rank-one SPCA in literature; 
	\item All our analyses of exact and approximation algorithms can be extended to the general SPCA \eqref{eq_spca}, where the results of approximation algorithms are displayed in  Table~\ref{table:spcaratio}; and
	\item The numerical study on large-scale instances (e.g., $m=16313, n=2365$) shows the high solution quality and scalability
	of our proposed approximation algorithms.
\end{enumerate}

\begin{table}[h] 
	\vskip -0.1in
	\centering
	\caption{Summary of Approximation Algorithms for SSVD \eqref{eq_ssvdcom}}
		\vskip -0.08in 
	\label{table:ratio}
	\begin{threeparttable}
		\setlength{\tabcolsep}{2pt}\renewcommand{\arraystretch}{1.2}
		\begin{tabular}{c| c |c| c| c | c }
			\hline  
			& \multicolumn{3}{c|}{Selection Algorithms} &	\multicolumn{2}{c}{Searching Algorithms } \\
			\hline
			Algorithm  & Algorithm~\ref{algo:trun1} &	Algorithm~\ref{algo:trun2} & Algorithm~\ref{algo:trun3} &	Algorithm~\ref{algo:svd_greedy}& Algorithm~\ref{algo:svd_localsearch}
			\\ 
			\hline
			Ratio &  $ {1}/{\sqrt{\min\{s_1, s_2\}}}$  &  ${1}/{\sqrt{k\min\{s_1, s_2\}}}$ & ${\sqrt{s_1s_2}}/{(k\sqrt{mn})}$ & ${1}/{\sqrt{ks_1s_2}}$ & ${1}/{\sqrt{ks_1s_2}}$  \\
			\hline
			\multirow{2}{*}{Complexity} &  \multirow{2}{*}{NP-hard}   &  $O((m+n)(m\log (m) $ & $O(m\log(m)$ & $O(\max\{s_1, s_2\}$ & $O(L/\delta ks_1s_2$  \\
			&  &  $+n\log(n) + ks_1s_2))$ & $+n\log(n)+mn)$ & $(m+n)ks_1s_2)$ & $(ns_1+ms_2))$  \\
			\hline
		\end{tabular}%
		\vspace{-.1em}
		\begin{tablenotes}
			\item[1] $L =$ encoding length of $\bm A$, and $\delta> 0$ is the strict improvement factor
		\end{tablenotes}  
	\end{threeparttable}
	\vspace{-.8em}
\end{table}
\begin{table}[h] 
	\centering
	\caption{Summary of Approximation Algorithms for SPCA \eqref{eq_spca}}
		\vskip -0.08in 
	\label{table:spcaratio}
	\begin{threeparttable}
		\setlength{\tabcolsep}{2pt}\renewcommand{\arraystretch}{1.2}
		\begin{tabular}{c| c |c| c| c | c }
			\hline  
			& \multicolumn{3}{c|}{Selection Algorithms} &	\multicolumn{2}{c}{Searching Algorithms } \\
			\hline
			Algorithm  & Algorithm~\ref{algo:trun1spca} &	Algorithm~\ref{algo:trun2spca} & Algorithm~\ref{algo:trun3spca} &	Algorithm~\ref{algo:spca_greedy}& Algorithm~\ref{algo:spca_localsearch}
			\\ 
			\hline
			Ratio &  $ {1}/{\sqrt{s}}$  &  ${1}/{\sqrt{ks}}$ & ${s}/{(kn)}$ & $k/s$ & $k/s$  \\
			\hline
			\multirow{1}{*}{Complexity} &  \multirow{1}{*}{NP-hard}   &  $O(n(n\log(n)+ks^2)) $ & $O(n\log(n)+n^2)$ & $O(nks^3)$ & $O(L/\delta nks^3)$  \\
			\hline
		\end{tabular}%
		\begin{tablenotes}
			\item[1] $L =$ encoding length of $\bm A$, and $\delta> 0$ is the strict improvement factor
		\end{tablenotes}  
	\end{threeparttable}
	\vspace{-.8em}
\end{table}

	\noindent\textit{Organization.}
The remainder of  this paper is organized as follows.  Section~\ref{sec:pre} presents important preliminary results. Section~\ref{sec:model} shows equivalent formulations for SSVD.  Section~\ref{sec:algo} describes approximation algorithms of SSVD with provable performance guarantees.  Section~\ref{sec:spca} extends the analysis to SPCA. Section~\ref{sec:bc} introduces the branch-and-cut algorithms for exactly solving SPCA and SSVD.
Section~\ref{sec:data} presents the numerical illustration of our proposed algorithms. Section~\ref{sec:conclusions} concludes the paper.

\noindent\textit{Notation.}
The following notation is used throughout the paper. We use bold lower-case letters
(e.g., $\bm x$) and bold upper-case letters (e.g., $\bm X$) to denote vectors and matrices, respectively, and use the corresponding non-bold letters (e.g., $x_i$, $X_{ij}$) to denote their components. For positive integers $n$ and $s\le n$, we let $[n] :=\{1,2,\cdots, n\}$, let $[s,n]=\{s, s+1, \cdots, n\}$,  let $\S_+^n$ denote set of all the $n\times n$ symmetric positive semidefinite matrices, and let $\bm I_n$ denote the $n \times n$ identity matrix. For any two positive integers $i,j$, we let $\bm 0_{i,j}$ denote a size $i\times j$ all-zeros matrix and let $\bm 0_{i}$ denote a size $i$ all-zeros vector,
and we let $\bm 1_{i,j}$ to denote a size $i\times j$ all-ones matrix and let $\bm 1_{i}$ to denote a size $i$ all-ones vector. For a matrix $\bm X \in \Re^{m\times n}$ with rank $r$, let $\bm X_{:, j}$ denote its $j$th column for $j \in [n]$, let $\bm X_{i, :}$ denote its $i$th row for each $i\in [m]$,  let $\sigma_i(\bm X)$ denote its $i$th largest singular value for each $i\in [r]$, let $||\bm X||_2$, $||\bm X||_*$, and $||\bm X||_{F}$ denote the induced two norm, nuclear, and Frobenius norm, and if $\bm{X}$ is symmetric, we let $\bm X \succeq 0$ denote the positive semi-definiteness of matrix $\bm X$ and let $\lambda_i(\bm X)$ denote its $i$th largest eigenvalue for each $i\in [r]$. For a matrix $\bm X \in \Re^{m \times n}$ and any two subsets $S_1 \subseteq [m]$ and $S_2 \subseteq [n]$, let $\bm X_{S_1, j}$ denote the $j$th column of submatrix $\bm X_{S_1, [n]}$ for all $j\in [n]$, and let $\bm X_{i, S_2}$ denote the $i$th row of submatrix $\bm X_{[m], S_2}$ for all $i\in [m]$.
For a vector $\bm x \in \Re^{ n}$ and any subset $S \subseteq [n]$, we let $||\bm x||_2$ denote the two norm, let $||\bm x||_{\infty}$ denote the infinite norm, and let $\bm x_{S}$ denote the subvector with entries indexed by $S$. Give a real number $x\in \Re$, we let $\text{sign}(x)$ denote its sign.
All vectors and matrices are properly sized so that algebraic operations can be carried through.
Additional notation will be introduced later as needed.

\section{Preliminaries} \label{sec:pre}

According to the definition of SVD, we observe that
\begin{observation} \label{obs1}
	For any matrix $\bm B \in \Re^{m\times n}$, suppose that $\bm B = \sum_{i\in [r]} d_i \bm u_i \bm v_i^{\top}$ where $\bm d \in \Re_+^r$ and  both $\{ \bm u_i\}_{i\in [r]}$ and $\{ \bm v_i\}_{i\in [r]}$ are orthogonal (but possibly not orthonormal) vectors, respectively. Then the SVD of matrix $\bm B$ can be represented as
	\begin{align*}
	\bm B: = \sum_{i\in [r]} d_i ||\bm u_i||_2 ||\bm v_i||_2 \frac{\bm u_i }{ ||\bm u_i||_2 }\frac{\bm v_i^{\top}}{||\bm v_i||_2} ,
	\end{align*}
	and $\{d_i ||\bm u_i||_2 ||\bm v_i||_2\}_{i\in [r]}$ are the singular values of matrix $\bm B$, and $\{\bm u_i/ ||\bm u_i||_2\}_{i\in [r]}$ and $\{\bm v_i/ ||\bm v_i||_2\}_{i\in [r]}$ are corresponding left- and right- orthonormal singular vectors.
\end{observation} 

We use the notion of Sylvester's Hadamard matrix to construct worst-case examples to prove the tightness of the approximation ratios of our approximation algorithms.
\begin{definition}[Sylvester's Hadamard Matrix]
	According to Sylvester's construction \cite{seberry2005some}, for every positive integer $t$, there exists a symmetric $2^t\times 2^t$ matrix whose entries are all equal to $1$ or $-1$ and columns are orthogonal to each other. This is known as Sylvester's Hadamard matrix of order $2^t$, denoted by $\bm H(t)$.
\end{definition}
To be specific,  Sylvester's Hadamard matrix can be constructed sequentially, i.e., 
$\bm H(t+1)=\begin{pmatrix}
\bm H(t) & \bm H(t)\\
\bm H(t)& -\bm H(t)
\end{pmatrix}$ with  $\bm H(1)=[1]$.
For simplicity, we refer to Sylvester's Hadamard matrix as the Hadamard matrix throughout this paper.

\section{Exact MISDP of SSVD} \label{sec:model}
In this section, we develop an exact mixed-integer semidefinite programming (MISDP) formulation for SSVD \eqref{ssvd}. The MISDP relies on the following key lemma that establishes a semidefinite representation of the Ky Fan $k$-norm for a given matrix, motivated by the variational formulation of the top $k$ singular vectors in  \cite{golub2013matrix} and the facts that the largest singular value and  the sum of squares of $k$ largest singular values of a  matrix have been  shown to be semidefinite representable in \cite{doan2016finding,li2020exact}.

\begin{lemma} \label{lem:svd}
	For any matrix $\bm B \in \Re^{m \times n}$ with rank $r$, suppose that $\bm B = \sum_{i\in [r]} \sigma_i(\bm B) \overline{\bm u}_i \overline{\bm v}_i^{\top}$ denotes the SVD of $\bm B$. Then  for a positive integer $k \le r$, 
	\begin{align*}
	\max_{\bm X \in \Re^{n \times m} }  \left\{\tr(\bm B \bm X): ||\bm X||_2 \le 1, ||\bm X||_{*}\le k  \right\} = ||\bm B||_{(k)},
	\end{align*}
	where the equality can be achieved by $\bm X^* =\sum_{ i \in [k]} \overline{\bm v}_i(\overline{\bm u}_i)^{\top} $.
\end{lemma}
\begin{proof}
	To show the identity, let us consider the following maximization problem
	\begin{align} \label{eq3}
	\max_{\bm X \in \Re^{n \times m} }  \left\{\tr(\bm B \bm X): ||\bm X||_2 \le 1, ||\bm X||_{*}\le k  \right\} .
	\end{align}
	
	First, for any feasible solution $(\bm U, \bm V)$  to the maximization problem   \eqref{tsvd}, we can construct a feasible solution $\bm X= \bm V \bm U^{\top}$ to problem \eqref{eq3} that yields the same objective value. Thus, the optimal value of problem \eqref{eq3} is lower bounded by $||\bm B||_{(k)}$.
	
	On the other hand,  suppose that an optimal solution $\bm X^*$ of problem \eqref{eq3} has rank $d\geq k$ and  its SVD is $\bm X^*  = \sum_{i\in [d]} \sigma_i(\bm X^*) \bm p_i \bm  q_i^{\top}$. As $\|\bm X^*\|_2\leq 1$, we have $\sigma_i(\bm X^*)\leq 1$ for each $i\in [d]$. Let $\delta$ denote a permutation of $[d]$ such that $|\bm q_{\delta(1)}^{\top}\bm B \bm p_{\delta(1)} |\ge \cdots \ge |\bm q_{\delta(d)}^{\top}\bm B \bm p_ {\delta(d)}|$. Thus, the optimal value of problem \eqref{eq3} is equal to 
	\begin{align*}
	\tr(\bm B\bm X^*)=\sum_{i\in [d]}\sigma_{\delta(i)} (\bm X^*)  \bm q_{\delta(i)}^{\top}\bm B \bm p_{\delta(i)} \le \sum_{i\in [k]}|\bm q_{\delta(i)}^{\top}\bm B \bm p_{\delta(i)}| \le ||\bm B||_{(k)},
	\end{align*}
	where the first inequality is from the fact that $0\leq \sigma_i(\bm X^*) \le 1$ for all $i\in [d]$ and $\sum_{i\in [d]} \sigma_i(\bm X^*) \leq k$ and the second one is because $\bm{U}=[\bm q_{\delta(1)}, \cdots, \bm q_{\delta(k)}] $ and $\bm{V}=[\text{sign}(\bm p_{\delta(1)}^{\top}\bm B \bm q_{\delta(1)})\bm p_{\delta(1)}, \cdots,$ $\text{sign}(\bm p_{\delta(k)}^{\top}\bm B \bm q_{\delta(k)})\bm p_{\delta(k)}] $ is feasible to the maximization problem  \eqref{tsvd} and yields  objective value $\sum_{i\in [k]}|\bm q_{\delta(i)}^{\top}\bm B \bm p_{\delta(i)}|$.
	
	Thus, problem \eqref{eq3} has optimal value $||\bm B||_{(k)}$ with an optimal solution $\bm X^*= \sum_{ i \in [k]} \overline{\bm v}_i(\overline{\bm u}_i)^{\top} $. \qed
\end{proof}

	Introducing binary variables to represent the set variable in SSVD \eqref{eq_ssvdcom} and using Lemma~\ref{lem:svd} to represent the objective function in SSVD \eqref{eq_ssvdcom}, we arrive at the exact MISDP \eqref{eq_misdp} in the following.

\begin{restatable}{theorem}{thmsdp}\label{thm_sdp}
	The SSVD \eqref{ssvd} can be formulated as 
	\begin{align}\label{eq_misdp}
	&\text{\rm (SSVD)} \quad  w^*:=\max_{
		\begin{subarray}{c}
		\bm z \in \{0,1\}^{m},
		\bm y \in \{0,1\}^n,
		\bm X \in \Re^{n \times m}
		\end{subarray}
	} \bigg\{\tr(\bm A \bm X): ||\bm X||_2 \le 1,    
	||\bm X||_{*} \le k, \notag \\ 
	&||\bm X_{:, i}||_{\infty} \le z_i, \forall i \in [m], \sum_{i \in [m]} z_i \le s_1,  
	||\bm X_{j,:}||_{\infty} \le y_j, \forall j \in [n], \sum_{j\in [n]} y_j \le s_2 \bigg\}.
	\end{align}
\end{restatable}

\begin{proof}
	We prove the equivalence of SSVD \eqref{eq_ssvdcom} and problem \eqref{eq_misdp} via the one-to-one solution correspondence.
	
	First, for any feasible solution $(\bm z, \bm y,\bm{X})$ to problem \eqref{eq_misdp}, let us denote by $S_1$ and $S_2$ the supports of $\bm z$ and $\bm y$, separately. Thus, $(S_1, S_2)$ is feasible to  SSVD \eqref{eq_ssvdcom} with the objective value $||\bm A_{S_1, S_2} ||_{(k)}$. We will show that its corresponding objective value of problem \eqref{eq_misdp} is also equal to $||\bm A_{S_1, S_2} ||_{(k)}$ below.
	
	According to the constraints in problem \eqref{eq_misdp}, we can derive that
	$X_{ij} = 0$, for all $([n]\times [m]) \setminus (S_2\times S_1)$. Thus, problem \eqref{eq_misdp} can reduce to the following maximization problem 
	\begin{align}\label{eq_part_ii}
	\max_{\bm X \in \Re^{|S_2| \times |S_1|}} \left\{\tr(\bm A_{S_1, S_2} \bm X_{S_2, S_1}):  ||\bm X_{S_2, S_1}||_2 \le 1, ||\bm X_{S_2, S_1}||_* \le k \right\},
	\end{align}
	which has the optimal value $||\bm A_{S_1, S_2} ||_{(k)}$ by Lemma~\ref{lem:svd}.
	
	Second, for any feasible solution $(S_1, S_2)$ to SSVD \eqref{eq_ssvdcom}, we let $z_i = 1$ if $i\in S_1$ and zero, otherwise for each $i\in [m]$ and  let $y_j = 1$ if $j \in S_2$ and zero, otherwise for each $j\in [n]$. Next, according to  Lemma~\ref{lem:svd}, let us denote $\bm X_{S_2, S_1}$ to be an optimal solution to problem \eqref{eq_part_ii} with the objective value equal to the Ky Fan $k$-norm of submatrix $\bm A_{S_1, S_2}$. We augment the matrix $\bm X_{S_2, S_1}$ to be of size $m\times n$ by filling zero entries. Thus, the constructed solution $(\bm z, \bm y,\bm{X})$ is feasible to problem \eqref{eq_misdp} with the same objective value $||\bm A_{S_1, S_2} ||_{(k)}$. 
	
	In conclusion, problem \eqref{eq_misdp} is an exact formulation of SSVD \eqref{eq_ssvdcom}. \qed
\end{proof}

We make the following remarks about the result of Theorem~\ref{thm_sdp}: (i) since both nuclear norm and induced two-norm are positive semidefinite representable \citep{ben2011lectures}, SSVD \eqref{eq_misdp} can be easily recast as an MISDP; (ii) Theorem~\ref{thm_sdp} provides an MISDP formulation for SSVD \eqref{ssvd}, which can be directly solved by off-the-shelf solvers such as YALMIP, Pajarito, and SCIP-SDP \citep{gally2018framework}; (iii) based on the MISDP \eqref{eq_misdp}, we can design an exact branch-and-cut algorithm for SSVD in Section~\ref{sec:bc} and in our numerical study, it works well for small- and medium-sized instances (e.g., it can solve the case of $m=n=500, s_1=s_2=5, k=2$ within one minute); and (iv) as an important corollary, we derive two equivalent MISDPs for SPCA \eqref{eq_spca}.

	\section{Approximation Algorithms of SSVD} \label{sec:algo}
The results from the previous section are useful to  solve small- or medium-sized instances to optimality; however, the exact algorithm might not be scalable to large-scale instances (see our numerical study section for more details). In this section, we develop and analyze three different selection algorithms and two searching algorithms of SSVD \eqref{eq_ssvdcom} with theoretical guarantees for solving the large-scale  instances to near optimality.

\subsection{Selection Algorithms }
We observe that the difficulty of the exact formulations arises from the fact that the representation of Ky Fan $k$-norm in SSVD \eqref{eq_ssvdcom} is either a nonconvex program or a difficult convex SDP. 
This motivates us to explore more effective selection criteria to replace the Ky Fan $k$-norm such that the obtained submatrix by alternative criteria is still a high-quality solution to SSVD \eqref{eq_ssvdcom}. 
Therefore, we design three different selection criteria that are related to the Ky Fan-$k$ norm but are much easier to compute, which correspond to three selection algorithms for SSVD  \eqref{eq_ssvdcom}. All  selection algorithms come with approximation ratios proven to be tight when $s_1\le m/2$ and $s_2\le n/2$.

\noindent\textbf{Selection Algorithm~\ref{algo:trun1}.} In the first algorithm, we select a submatrix of size at most $s_1 \times s_2$ that maximizes the Frobenius norm.
The detailed implementation can be found in Algorithm~\ref{algo:trun1}. We prove a $1/\sqrt{\min\{s_1, s_2\}}$-approximation ratio  of Algorithm~\ref{algo:trun1} in Theorem~\ref{thms1}, and this ratio is unimprovable. Note that both the selection procedure and approximation ratio of Algorithm~\ref{algo:trun1} is independent of the parameter $k$.

\begin{algorithm}[htb]
	\caption{Selection Algorithm I for SSVD \eqref{eq_ssvdcom}}
	\label{algo:trun1}
	\begin{algorithmic}[1]
		\State  {\bfseries Input:} A matrix $\bm A \in \Re^{m \times n}$,  integers   $s_1\in [m]$, $s_2\in [n]$
		\State  Suppose $(\hat{S}_1, \hat{S}_2)$ is an optimal solution to the problem below
		\begin{align} \label{eq_alter}
		\ \max_{
			\begin{subarray}{c}
			S_1\subseteq[m],  S_2\subseteq [n]
			\end{subarray}
		} \left\{||\bm A_{S_1, S_2}||_F: |S_1| \le s_1,|S_2| \le s_2\right\}.
		\end{align}
		\State  {\bfseries Output:} $(\hat{S}_1, \hat{S}_2)$
	\end{algorithmic}
\end{algorithm}
\vspace{-0.05in}

\begin{restatable}{theorem}{thmsone} \label{thms1} 
	The selection Algorithm~\ref{algo:trun1} yields a $1/\sqrt{\min\{s_1, s_2\}}$ approximation ratio for SSVD \eqref{eq_ssvdcom}, i.e., suppose that the output of Algorithm~\ref{algo:trun1} is $(\hat{S}_1, \hat{S}_2)$, then
	\begin{align*}
	||\bm A_{\hat{S}_1, \hat{S}_2}||_{(k)} \ge \frac{1}{\sqrt{\min\{s_1, s_2\}}}  w^*,
	\end{align*}
	and  the ratio is tight (i.e., the best ratio that one can prove) {when $s_1\le m/2$ and $s_2\le n/2$}.
\end{restatable}


\begin{proof}
	Let us define $\hat{s}:= \min\{s_1, s_2\}$, denote an optimal solution of SSVD \eqref{eq_ssvdcom} by $(S_1^*, S_2^*)$, and define the singular values of the optimal submatrix $\bm A_{S_1^*, S_2^*}$ to be $\sigma_1^* \ge\cdots \ge \sigma_{\hat{s}}^* \ge 0$.
	Correspondingly, for the output $(\hat{S}_1, \hat{S}_2)$ of Algorithm~\ref{algo:trun1}, we denote the singular values of the submatrix $\bm A_{\hat{S}_1, \hat{S}_2}$ to be $\hat{\sigma}_1 \ge\cdots \ge \hat{\sigma}_{\hat{s}} \ge 0$.
	
	Since $(\hat{S}_1, \hat{S}_2)$ is feasible to SSVD \eqref{eq_ssvdcom}, its objective value is equal to 
	\begin{align*}
	& ||\bm A_{\hat{S}_1, \hat{S}_2}||_{(k)}  = \sum_{i\in [k]} \hat{\sigma}_i \ge \sqrt{\sum_{i\in [k]} \hat{\sigma}_i^2} \ge \sqrt{\frac{k}{\hat{s}} \sum_{i\in [\hat{s}]} \hat{\sigma}_i^2}
	= \sqrt{\frac{k}{\hat{s}}} ||\bm A_{\hat{S}_1, \hat{S}_2}||_F 
	\ge \sqrt{\frac{k}{\hat{s}}} ||\bm A_{{S}_1^*, {S}_2^*}||_F  \\
	&= \sqrt{\frac{k}{\hat{s}}} \sqrt{\sum_{i\in [\hat{s}]} (\sigma_i^*)^2} 
	\ge \sqrt{\frac{k}{\hat{s}}} \sqrt{\sum_{i\in [k]} (\sigma_i^*)^2} \ge \frac{1}{\sqrt{\hat{s}}} \sum_{i\in [k]}\sigma_i^* = \frac{1}{\sqrt{\hat{s}}}  w^*,
	\end{align*}	
	where the first and fifth inequalities are because of Cauchy-Schwarz inequality, the second and third equalities  are due to the definition of the Frobenius norm, the third inequality is due to the optimality of the solution $(\hat{S}_1, \hat{S}_2)$ to problem \eqref{eq_alter}, and the forth inequality is due to $k\leq \hat{s}$.

	The ratio is tight when $s_1\le m/2$  and $s_2\le n/2$ since the following example can achieve it.
	\begin{example} \label{eg1}
		Suppose that $k=s_1 < s_2$, $m =2s_1$, $n= 2s_2$ and matrix $\bm A \in \Re^{m \times n}$ is defined by
		\begin{align*}
		\bm A = \begin{pmatrix}
		\bm 1_{s_1, s_2} & \bm 0_{s_1, s_2}\\
		\bm 0_{s_1, s_2} & \bm D
		\end{pmatrix}, \bm D \in \Re^{s_1\times s_2}:= \sum_{i\in [k]} \frac{s_1\sqrt{s_2}}{k} \bm e_i \bm f_i^{\top},
		\end{align*}
		where $\bm e_i \in \Re^{s_1}$ is the $i$th column of the identity matrix $\bm I_{s_1}$ and $\bm f_i \in \Re^{s_2}$ is the $i$th column of the identity matrix $\bm I_{s_2}$.
	\end{example}
	In Example~\ref{eg1},	the Algorithm~\ref{algo:trun1} returns a solution $\hat{S}_1 = [s_1]$ and $\hat{S}_2 = [s_2]$ since $||\bm 1_{s_1, s_2}||_F= ||\bm D||_F $, and the objective value of  SSVD \eqref{eq_ssvdcom}  is $||\bm 1_{s_1, s_2}||_{(k)}=\sqrt{s_1s_2}$. However, the true optimal value is $||\bm D||_{(k)}=s_1\sqrt{s_2}$. This shows that $1/\sqrt{\min\{s_1,s_2\}}$-approximation ratio is achievable. \qed
\end{proof}

The computational difficulty of selection Algorithm~\ref{algo:trun1} lies in how to effectively solve  problem \eqref{eq_alter} at Step 2. In fact,  problem \eqref{eq_alter} can be shown to be NP-hard. Fortunately, we prove that the problem \eqref{eq_alter} can be formulated as a mixed-integer linear program (MILP), making it much more tractable. {In addition, problem \eqref{eq_alter} can be easily solved when $\bm A$ is rectangular diagonal, $s_1=1$, or $s_2=1$.}

\begin{restatable}{proposition}{propsone} \label{props1}
	The problem \eqref{eq_alter} is NP-hard and can be formulated as  the following MILP
	\begin{align*}
	\max_{
		\begin{subarray}{c}
		\bm z\in [0,1]^{m \times n},\\
		\bm x \in \{0,1\}^m, \bm y \in \{0,1\}^n
		\end{subarray} } \bigg\{ \sum_{i\in [m]} \sum_{j\in [n]} A_{ij}^2 z_{ij}: 
	z_{ij}\le x_i, z_{ij}\le y_j ,\forall i\in [m], \forall j\in [n], \sum_{i\in [m]} x_i \le s_1,   \sum_{j\in [n]} y_j \le s_2 \bigg\}.
	\end{align*}
\end{restatable}
\begin{proof}
	See Appendix \ref{proof:props1}. \qed
\end{proof}

\textbf{Selection Algorithm~\ref{algo:trun2}.} The combinatorial optimization problem \eqref{eq_alter} in Algorithm~\ref{algo:trun1} can be difficult since it jointly selects the two subsets $(\hat{S}_1, \hat{S}_2)$. To be more scalable, we propose another selection algorithm that selects rows and columns in a separable manner and then combines them  to form a submatrix selection for SSVD \eqref{eq_ssvdcom}. 

Specifically, given a data matrix $\bm A$, our selection Algorithm~\ref{algo:trun2} runs as follows: (a) first, for each $j\in [n]$, we compute a size-$s_1$ subset of rows $\hat{S}_1^j \subseteq [m]$ corresponding to the indices of the $s_1$ largest absolute entries in $j$th column $\bm A_{:,j}$, and we also let the subset of columns $\hat{S}_2^j \subseteq [n]$ denote the index set of $s_2$ columns having the largest $L_2$ norm from submatrix $\bm A_{\hat{S}_1^j, [n]}$, and form a feasible solution $(\hat{S}_1^j , \hat{S}_2^j )$ to SSVD \eqref{eq_ssvdcom}; (b) next, we perform the same selection procedure on matrix $\bm A$ by swapping the order of row and column operations and obtain another $m$ feasible solutions to SSVD \eqref{eq_ssvdcom}; (c) finally, Algorithm~\ref{algo:trun2} returns the one with the largest Ky Fan $k$-norm among all the $m+n$ candidate solutions. 

Different from Algorithm~\ref{algo:trun1}, we see  that the approximation ratio of the selection Algorithm~\ref{algo:trun2} is proportional to $1/\sqrt{k}$ and its time complexity also depends on $k$, as shown in Theorem~\ref{thms2}. 
We further show that this approximation ratio is tight by providing an instance that can achieve it.

\begin{algorithm}[htb]
	\caption{Selection Algorithm for SSVD \eqref{eq_ssvdcom}}
	\label{algo:trun2}
	\begin{algorithmic}[1]
		\State {\bfseries Input:} A matrix $\bm A \in \Re^{m \times n}$,  integers   $s_1\in [m]$, $s_2\in [n]$, and $1 \le k \le \min\{s_1, s_2\}$
		\For{$j\in [n]$}
		\State  Let $\hat{S}_1^j $ denote the index set of the $s_1$ largest absolute entries in $\bm A_{:,j}$
		\State  
		$
		\hat{S}_2^j \in \argmax_{S\subseteq [n]} \{\sum_{\ell \in S} ||\bm A_{\hat{S}_1^j, {\ell} }||_2: |S|=s_2 \}
		$		
		\EndFor
		
		\For{$i \in [m]$}
		\State Let $\hat{T}_2^i =$ denote the index set of the $s_2$ largest absolute entries in $\bm A_{i,:}$
		\State 
		$
		\hat{T}_1^i \in \argmax_{T\subseteq [m]} \{\sum_{\ell \in T} ||\bm A_{{\ell}, \hat{T}_2^i }||_2: |T|=s_1 \}
		$		
		\EndFor
		
		\State  {\bfseries Output:} The best solution for SSVD \eqref{eq_ssvdcom} among the candidate set $\{(\hat{S}_1^j, \hat{S}_2^j)\}_{j\in [n]} \cup\{(\hat{T}_1^i, \hat{T}_2^i)\}_{i\in [m]} $
	\end{algorithmic}
\end{algorithm}


\begin{restatable}{theorem}{thmstwo} \label{thms2} 
	The selection Algorithm~\ref{algo:trun2} yields a $1/\sqrt{k\min\{s_1, s_2\}}$ approximation ratio for SSVD \eqref{eq_ssvdcom}, i.e., suppose that the output of  Algorithm~\ref{algo:trun2} is $(\hat{S}_1, \hat{S}_2)$, then
	\begin{align*}
	||\bm A_{\hat{S}_1, \hat{S}_2}||_{(k)} \ge \frac{1}{\sqrt{k\min\{s_1, s_2\}}}  w^*,
	\end{align*}
	the ratio is tight  {when $s_1\le m/2$ and $s_2\le n/2$}, and the time complexity of  Algorithm~\ref{algo:trun2} is $O((m+n)(m\log (m)+ n\log(n) + ks_1s_2))$.
\end{restatable}

\begin{proof}
	The proof can be divided into three parts: derivation and tightness of the approximation ratio and derivation of the time complexity.

	\noindent\textbf{Deriving the Approximation Ratio.} 	
	We first show a technical result that for any matrix $\bm B \in \Re^{m\times n}$, its largest singular value must satisfy
	\begin{align} \label{lem:sigma1}
	\sigma_{1}\left(\bm B\right) & = \max \left(\sqrt{\lambda_1(\bm B \bm B^{\top})}, \sqrt{\lambda_1(\bm B^{\top} \bm B)} \right)  \notag \ge \max\left\{ \max_{i\in [m]} \sqrt{\left(\bm B \bm B^{\top}\right)_{ii}} ,\ \  \max_{j\in [n]}\sqrt{ \left(\bm B^{\top} \bm B\right)_{jj}}\right\} \notag \\
	&= \max \left\{ \max_{i\in [m]} ||\bm B_{i, :}||_2, \max_{j\in [n]} ||\bm B_{:, j}||_2\right\},
	\end{align}
	where the first equality is because  $\sigma_{1}^2(\bm B)$ is the largest eigenvalue of both symmetric matrices $\bm B \bm B^{\top}$ and $\bm B^{\top} \bm B$, 
	the second inequality is because the largest eigenvalue of any symmetric matrix must be no less than its largest diagonal element, and the third equality results from the fact that for matrix $\bm B \bm B^{\top}$ (or $\bm B^{\top} \bm B)$), the $i$th (or $j$th) diagonal element is equal to the two-norm of the $i$th row (or $j$th column) of matrix $\bm B$ for all $i\in [m]$ (or $j\in [n]$).
	
	Then let us consider the following two combinatorial optimization problems
	\begin{align}\label{eq_opt}
	(\bar{S}, j^*)  \in \max_{ S\subseteq [m], j\in [n]} \{||\bm A_{S, j}||_2: |S|=s_1\}, (\bar{T}, i^*) \in \max_{ S\subseteq [n], i\in [m]} \left\{||\bm A_{i, S}||_2: |S|=s_2\right\}.
	\end{align}
	
	\begin{subequations}
		Next, we will prove the following claim.
		\begin{claim}\label{claim1}
			For the output $(\hat{S}_1, \hat{S}_2)$ of Algorithm~\ref{algo:trun2} , we have 
			$$	||\bm A_{\hat{S}_1, \hat{S}_2}||_{(k)}  \ge \max \{||\bm A_{\bar{S}, j^*}||_2,  ||\bm A_{ i^*, \bar{T}}||_2\}.$$
		\end{claim}
		\begin{proof}
			For the $j^*$-th column of matrix $\bm A$, Algorithm~\ref{algo:trun2} selects subsets $(\hat{S}_1^{j^*}, \hat{S}_2^{j^*})$ at Steps 3 and 4 and the selected two subsets must satisfy
			\begin{align} \label{ineq1}
			\max_{\ell \in \hat{S}_2^{j^*}} ||\bm A_{\hat{S}_1^{j^*}, \ell}||_2 = \max_{\ell \in [n]} ||\bm A_{\hat{S}_1^{j^*}, \ell}||_2 \ge ||\bm A_{\hat{S}_1^{j^*}, j^*}||_2 \ge ||\bm A_{\bar{S}, j^*}||_2, 
			\end{align}
			where the first equality is due to the optimality of $\hat{S}_2^{j^*}$ at Step 4, and the second inequality results from the fact that $\hat{S}_1^{j^*}$ denotes the index set of the $s_1$ largest absolute entries in $\bm A_{:,j^*}$.
			
			Similarly, for the corresponding subsets $(\hat{T}_1^{i^*}, \hat{T}_2^{i^*})$, we can show that 
			\begin{align} \label{ineq2}
			\max_{\ell \in \hat{T}_1^{i^*}} ||\bm A_{ \ell,  \hat{T}_2^{i^*}}||_2 \ge ||\bm A_{ i^*, \bar{T}}||_2.
			\end{align}
			
			Since the output $(\hat{S}_1, \hat{S}_2)$ of Algorithm~\ref{algo:trun1} is the best solution among the candidates $\{(\hat{S}_1^j, \hat{S}_2^j)\}_{j\in [n]} \cup\{(\hat{T}_1^i, \hat{T}_2^i)\}_{i\in [m]} $, the output objective value must satisfy
			\begin{align*}
			&||\bm A_{\hat{S}_1, \hat{S}_2}||_{(k)} \ge \max_{{S}_1, {S}_2} \left\{ \sigma_{1}(\bm A_{{S}_1, {S}_2}): ({S}_1, {S}_2) \in \{(\hat{S}_1^j, \hat{S}_2^j)\}_{j\in [n]} \cup\{(\hat{T}_1^i, \hat{T}_2^i)\}_{i\in [m]}  \right\}\notag \\
			&\ge \max_{{S}_1, {S}_2} \bigg\{\max \left\{  \max_{j\in S_2} ||\bm A_{{S}_1, j}||_2, \max_{i\in S_1} ||\bm A_{i, {S}_2}||_2\right\}: ({S}_1, {S}_2) \in \{(\hat{S}_1^j, \hat{S}_2^j)\}_{j\in [n]} \cup\{(\hat{T}_1^i, \hat{T}_2^i)\}_{i\in [m]}  \bigg \} \notag \\
			&\ge \max \bigg\{\max_{\ell \in \hat{S}_2^{j^*}} ||\bm A_{\hat{S}_1^{j^*}, \ell}||_2,  \max_{\ell \in \hat{T}_1^{i^*}} ||\bm A_{ \ell,  \hat{T}_2^{i^*}}||_2 \bigg\} \ge \max \{|| \bm A_{\bar{S}, j^*}||_2,  ||\bm A_{ i^*, \bar{T}}||_2\},
			\end{align*}
			where the first inequality is due to the output solution having the largest objective value and $ ||\bm A_{{S}_1, {S}_2} ||_{(k)}\ge \sigma_{1}(\bm A_{{S}_1, {S}_2}) $, the second one is due to \eqref{lem:sigma1},  the third one is from the feasibility of two solutions $(\hat{S}_1^{j^*}, \hat{S}_2^{j^*})$ and $(\hat{T}_1^{i^*}, \hat{T}_2^{i^*})$, and the last one is from the inequalities \eqref{ineq1} and \eqref{ineq2}. \qedA
		\end{proof}
	\end{subequations}
	
	Now we let $(S_1^*, S_2^*)$ denote an optimal solution to SSVD \eqref{eq_ssvdcom}. According to Claim \ref{claim1}, the optimal value is upper bounded by
	\begin{align*}
	&w^* = ||\bm A_{S_1^*, S_2^*}||_{(k)} \le \sqrt{k}\sqrt{\sum_{ i \in [k]} \sigma_{i}^2(\bm A_{S_1^*, S_2^*})} \le \sqrt{k}|| \bm A_{S_1^*, S_2^*} ||_F =  \sqrt{k} \sqrt{\sum_{ i \in S_1^*} \sum_{j\in S_2^*} A_{ij}^2} \\
	&\le
	\begin{cases}
	\sqrt{ks_1} \max_{i\in S_1^*} ||\bm A_{i, S_2^*}||_2 \le  \sqrt{ks_1}  ||\bm A_{ i^*, \bar{T}}||_2 \le \sqrt{ks_1} ||\bm A_{\hat{S}_1, \hat{S}_2}||_{(k)} \\
	\sqrt{ks_2} \max_{j\in S_2^*}  ||\bm A_{ S_1^*, j}||_2  \le  \sqrt{ks_2} ||\bm A_{\bar{S}, j^*}||_2 \le\sqrt{ks_2}  ||\bm A_{\hat{S}_1, \hat{S}_2}||_{(k)}
	\end{cases}
	,
	\end{align*}
	where the fourth inequality in both lines is from the definitions of $(\bar{T}, i^*)$ and $(\bar{S}, j^*)$  in \eqref{eq_opt} and the last one on both lines is from Claim \ref{claim1}.
	Thus, 	we obtain a  $1/\sqrt{k\min\{s_1,s_2\}}$-approximation ratio.
	\vspace{10pt}
	
	\noindent\textbf{Tightness of the Approximation Ratio.} 	
	A worst-case example can be constructed to show that the approximation ratio of Algorithm~\ref{algo:trun2} is tight.
	\begin{example}\label{eg2}
		For any integers $k>0,c>0, t>0$, suppose that $s_1 = ck$, $s_2 = 2^{t}$ such that $\min\{s_1, s_2\} \ge k$, and $m =(2+ s_2)s_1$,  $n = (2+s_1)s_2$. Then we construct a matrix $\bm A$ as below
	\end{example}
	\begin{align*}
	\bm A = \begin{pmatrix}
	\bm A_{1} & \bm 0_{s_1, s_2} & \bm A_3 \\
	\bm 0_{s_1, s_1\times s_2}  & \bm 0_{s_1, s_2}  & \bm 0_{s_1, s_2} \\
	\bm 0_{s_1\times s_2, s_1\times s_2}  & \bm 0_{s_1\times s_2, s_2}  & \bm A_{2} 
	\end{pmatrix}\in \Re^{m\times n},  
	\end{align*}
	where 
	\begin{align*}
	\bm A_{1}  = \begin{pmatrix}
	\bm 1_{s_2}^{\top} & \bm 0_{s_2}^{\top} & \cdots & \bm 0_{s_2}^{\top}\\
	\bm 0_{s_2}^{\top} & \bm 1_{s_2}^{\top} & \cdots & \bm 0_{s_2}^{\top}\\
	\cdots \\
	\bm 0_{s_2}^{\top} & \bm 0_{s_2}^{\top} & \cdots & \bm 1_{s_2}^{\top}\\
	\end{pmatrix} \in \Re^{s_1, s_1\times s_2}, \bm A_{2} = \begin{pmatrix}
	\bm 1_{s_1} & \bm 0_{s_1} & \cdots & \bm 0_{s_1}\\
	\bm 0_{s_1} & \bm 1_{s_1} & \cdots & \bm 0_{s_1}\\
	\cdots \\
	\bm 0_{s_1} & \bm 0_{s_1} & \cdots & \bm 1_{s_1}
	\end{pmatrix}\in \Re^{s_1\times s_2, s_2} ,
	\end{align*}
	\begin{align*}
	\bm A_{3}  = \sum_{i\in[k]} \bm P_{:,i} \bm H_{:,i}^{\top}\in \Re^{s_1, s_2},   \bm P = \begin{pmatrix}
	\bm 1_{s_1/k} & \bm 0_{s_1/k} & \cdots & \bm 0_{s_1/k}\\
	\bm 0_{s_1/k} & \bm 1_{s_1/k} & \cdots & \bm 0_{s_1/k}\\
	\cdots \\
	\bm 0_{s_1/k} & \bm 0_{s_1/k} & \cdots & \bm 1_{s_1/k}
	\end{pmatrix}\in \Re^{s_1, k} ,
	\end{align*}
	and $\bm H \in \Re^{s_2, s_2}$ is a Hadamard matrix of order $s_2= 2^{t}$.
	
	Note that in Example~\ref{eg2}, the largest absolute value among all the entries in matrix $\bm A$ is one. Next, we check the selection Algorithm~\ref{algo:trun2} as follows: 
	\begin{enumerate}[(a)]
		\item For each $j\in [s_1\times s_2]$, there is only one entry with value 1 in the $j$th column of $\bm A$ and the location of the entry is at $(\lceil j/s_2\rceil,j)$. Therefore, we can select $\hat{S}_1^j = \{\lceil j/s_2\rceil\}\cup [s_1+2, 2s_1]$; for the  row submatrix $\bm A_{\hat{S}_1^j, [n]}$, the largest two-norm among all of its columns is one. Thus, we can select $\hat{S}_2^j = [(\lceil j/s_2\rceil-1)\times s_2+1, \lceil j/s_2\rceil\times s_2]$.  The resulting objective value is equal to $\sqrt{s_2}$.
		\item For each $j\in [s_1\times s_2+1, (s_1+1)\times s_2]$, all the entries in the $j$th column of $\bm A$ are 0. Thus, we can select $\hat{S}_1^j = [s_1+1, 2s_1]$ and $\hat{S}_2^j = [s_1\times s_2+1, s_1\times s_2+ s_2]$. The resulting objective value is equal to 0.
		\item For each $j\in [(s_1+1)\times s_2+1, (s_1+2)\times s_2]$, there is a size-$s_1$ vector with all ones in the column $(\bm A_2)_{:,j-(s_1+1)\times s_2}$. Thus, we can select $\hat{S}_1^j = [2s_1+ (\ell-1)\times s_1+1, 2s_1+ \ell\times s_1]$ where $\ell = j-(s_1+1)\times s_2$ and $\hat{S}_2^j = [j-s_2+1, j]$. The resulting objective value is equal to $\sqrt{s_1}$.
	\end{enumerate}
	Similarly, by switching the row and column selections, we can obtain another $m$ feasible solutions to SSVD \eqref{eq_ssvdcom} and the resulting objective value can be $\sqrt{s_1}$, $\sqrt{s_2}$, or 0.
	Thus, the best objective value output by the selection Algorithm~\ref{algo:trun2} is $\max\{\sqrt{s_1}, \sqrt{s_2}\}$. However, for Example~\ref{eg2}, the optimal solution to SSVD \eqref{eq_ssvdcom} is the submatrix $\bm A_{3}$ with optimal value 
	\begin{align*}
	||\bm A_{3}||_{(k)} = \bigg|\bigg|\sum_{i\in[k]} \bm P_{:,i} \bm H_{:,i}^{\top}\bigg|\bigg|_{(k)} = \sum_{ i \in [k]} || \bm P_{:,i}||_2 ||\bm H_{:,i}||_2 = k\sqrt{s_1/k}\sqrt{s_2} = \sqrt{ks_1s_2},
	\end{align*}
	where the second equality is due to the fact that  $\{\bm P_{:,i}\}_{i\in [k]}$ and $\{\bm H_{:,i}\}_{i\in [k]}$   are orthogonal vectors, respectively and Observation \ref{obs1}.
	Thus, the approximation ratio of Example~\ref{eg2} is
	\begin{align*}
	\frac{\max\{\sqrt{s_1}, \sqrt{s_2}\} }{\sqrt{ks_1s_2}} = \frac{1}{\sqrt{k \min\{s_1, s_2\}}},
	\end{align*}
	which verifies  the tightness of the approximation ratio when $s_1 \le m/2$ and $s_2\le  n/2$.
	\vspace{10pt}
	
	\noindent\textbf{Deriving the Time Complexity.}  Steps 3 and 4 can be solved by sorting with the time complexity of $m\log(m)$ and $n\log(n)$, respectively. Thus, it takes $O(n(m\log (m)+ n\log(n)) )$ to implement the first for-loop. Similarly, it takes $O(m(m\log (m)+ n\log(n)) )$ to implement the second for-loop.
	
	It has been shown in \cite{shishkin2019fast} that for any $s_1\times s_2$ matrix, the computation of its $k$-truncated SVD can be done in $O(ks_1s_2)$. Finally, there are $m+n$ candidate solutions to be evaluated. This concludes the  time complexity analysis.\qed
\end{proof}

%

\noindent\textbf{Selection Algorithm~\ref{algo:trun3}.} In this selection algorithm, we trim the top left- and right- singular vectors of matrix $\bm A$, as presented in Algorithm~\ref{algo:trun3}. Notably, we select the submatrix  by 
specifying a size-$s_1$ subset of rows that corresponds to the $s_1$ largest absolute entries of the top left eigenvector, 
trimming the  largest left eigenvector accordingly to be of support size $s_1$, and then selecting a size-$s_2$ subset of columns that corresponds to the largest absolute entries of the product of the trimmed left eigenvector (augmented with zeros if necessary) with the original data matrix $\bm{A}$. We repeat the same procedure for the top right eigenvector. Since there are two candidate feasible solutions, we choose the one with a larger objective value of SSVD \eqref{eq_ssvdcom}. 

In Theorem~\ref{thms3}, we show that selection Algorithm~\ref{algo:trun3} yields a $\sqrt{s_1s_2}/(k\sqrt{mn})$-approximation ratio for SSVD \eqref{eq_ssvdcom} with time complexity of $O(m\log(m)+n\log(n)+mn)$. Though Algorithm~\ref{algo:trun3} has the most efficient implementation among the three selection algorithms, its performance is the worst as $k$ increases. In particular, this phenomenon is also observed in our numerical study.

\begin{algorithm}[htb] 
	\caption{Selection Algorithm for SSVD \eqref{eq_ssvdcom}} \label{algo:trun3}
	\begin{algorithmic}[1]
		\State \textbf{Input:} A matrix $\bm A \in \Re^{m \times n}$,  integers   $s_1\in [m]$, $s_2\in [n]$, and  $1\le k \le \min\{s_1, s_2\}$
		
		\State Compute  left- and right-singular vectors  $(\bm u_1, \bm v_1)$ corresponding to the largest singular value of $\bm A$
		
		\State Denote by $ \hat{S}_1$ the indices of $s_1$ largest absolute entries in $\bm u_1$
		\State Denote by $ \hat{S}_2$ the indices of $s_2$ largest absolute entries in $(\bm u_1)_{\hat{S}_1}^{\top} \bm A_{\hat{S}_1, [n]} $
		
		\State Denote by $ \hat{T}_2$ the indices  of $s_2$ largest absolute entries in $\bm v_1$
		\State Denote by $ \hat{T}_1$ the indices of $s_1$ largest absolute entries in $ \bm A_{[m], \hat{T}_2} (\bm v_1)_{\hat{T}_2} $
		
		
		\State	\textbf{Output:} The better solution of $(\hat{S}_1, \hat{S}_2)$ and $(\hat{T}_1, \hat{T}_2)$
	\end{algorithmic}
\end{algorithm}
\vspace{-0.15in}

\begin{restatable}{theorem}{thmsthree} \label{thms3}
	The selection Algorithm~\ref{algo:trun3} yields a $\sqrt{s_1s_2}/(k\sqrt{mn}) $ approximation ratio for SSVD \eqref{eq_ssvdcom}, i.e., suppose that the output of Algorithm~\ref{algo:trun3} is $(\hat{S}_1, \hat{S}_2)$, then
	\begin{align*}
	||\bm A_{\hat{S}_1, \hat{S}_2}||_{(k)} \ge \frac{\sqrt{s_1 s_2}}{k\sqrt{mn} } w^*,
	\end{align*}
	the ratio is tight {when $s_1\le m/2$ and $s_2\le n/2$}, and the time complexity of the selection Algorithm~\ref{algo:trun3} is $O(m\log (m)+ n\log(n) + mn)$.
\end{restatable}

\begin{proof}
	First, for the output $(\hat{S}_1, \hat{S}_2)$ from Algorithm~\ref{algo:trun3}, let us define the vectors $\hat{\bm u}_1 \in \Re^m$ and $\hat{\bm v}_1 \in \Re^n$ as follows
	\begin{align*}
	(\hat{u}_1)_i = \begin{cases}
	(u_1)_i, & \text{if } i \in \hat{S}_1\\
	0, & \text{otherwise} 
	\end{cases}  \forall i\in [m], \hat{\bm u}_1 = \frac{\hat{\bm u}_1}{||\hat{\bm u}_1||_2},   	(\hat{v}_1)_i = \begin{cases}
	(\hat{\bm u}_1^{\top} \bm A)_{i}, & \text{if } j \in \hat{S}_2\\
	0, & \text{otherwise} ,
	\end{cases} \forall j\in [n], \hat{\bm v}_1 = \frac{\hat{\bm v}_1}{||\hat{\bm v}_1||_2}.
	\end{align*}
	
	Then we have
	\begin{align*}
	\sqrt{\frac{s_1}{m}} \sigma_{1}(\bm A) = 	\sqrt{\frac{s_1}{m}} \sigma_{1}(\bm A) ||\bm u_1||_2 \le \sigma_{1}(\bm A) \hat{\bm u}_1^{\top} \bm u_1 =   \hat{\bm u}_1^{\top} \bm A \bm v_1 \le || \hat{\bm u}_1^{\top} \bm A||_2 \le \sqrt{\frac{n}{s_2}} \hat{\bm u}_1^{\top} \bm A \hat{\bm v}_1,
	\end{align*}
	where the first inequality is because $\hat{\bm u}_1^{\top} \bm u_1 = \sqrt{\sum_{i\in \hat{S}_1}(u_1)_i^2} \ge \sqrt{s_1/m}\sqrt{\sum_{i\in [m]}(u_1)_i^2}$,  the second equality is from $\bm A \bm v_1 = \sigma_{1}(\bm A)\bm u_1$, the second inequality is from Cauchy-Schwarz inequality and $||\bm v_1||_2 =1$, and the last one follows the similar reasoning as the first inequality.
	
	Now we have
	\begin{align*}
	w^* \le ||\bm A||_{(k)} \le k\sigma_{1}(\bm A ) \le k \frac{\sqrt{mn}}{\sqrt{s_1s_2}} \hat{\bm u}_1^{\top} \bm A \hat{\bm v}_1 = k \frac{\sqrt{mn}}{\sqrt{s_1s_2}} (\hat{\bm u}_1)_{\hat{S}_1}^{\top} \bm A_{\hat{S}_1, \hat{S}_2} (\hat{\bm v}_1)_{\hat{S}_2} \le k \frac{\sqrt{mn}}{\sqrt{s_1s_2}} \sigma_{1}(\bm A_{\hat{S}_1, \hat{S}_2}),
	\end{align*}
	where the first equality is because $ (\hat{ u}_1)_i = 0$ for all $i\in [m]\setminus\hat{S}_1$ and $ (\hat{ v}_1)_j = 0$ for all $j\in [n]\setminus\hat{S}_2$ and the last inequality is due to Part (i) in Lemma~\ref{lem:svd} with matrix $\bm A_{\hat{S}_1, \hat{S}_2}$ and  $k=1$. 
	
	Similarly, for the output $(\hat{T}_1, \hat{T}_2)$ from Algorithm~\ref{algo:trun3}, we also have
	\begin{align*}
	w^* \le k \frac{\sqrt{mn}}{\sqrt{s_1s_2}} \sigma_{1}(\bm A_{\hat{T}_1, \hat{T}_2}),
	\end{align*}
	This concludes that the approximation ratio of selection Algorithm~\ref{algo:trun3} is equal to $\sqrt{s_1s_2}/(k\sqrt{mn}) $. 
	
	For	the following example,  the output of the selection Algorithm~\ref{algo:trun3} achieves the ratio of $\sqrt{s_1s_2}/(k\sqrt{mn})$.
	
	\begin{example} \label{eg3}
		For any positive  integers $k>0, t_1>0, t_2>0$, suppose that $s_1 = 2^{t_1}$, $s_1 = 2^{t_2}$ such that $\min\{s_1, s_2\}\ge k+1$, and $m=2s_1$, $n=2s_2$, then the matrix $\bm A \in \Re^{n\times m}$ is constructed by
		\begin{align*}
		\bm A =  \sum_{i \in [k]}  \frac{\sqrt{mn}}{\sqrt{s_1s_2}} \begin{pmatrix}
		\bm 0_{s_1} \\
		\bm H_{:,i}
		\end{pmatrix}  \begin{pmatrix}
		\bm 0_{s_2} \\
		\bm F_{:,i}
		\end{pmatrix}^{\top} + \begin{pmatrix}
		\bm 1_{s_1} \\
		\bm H_{:,k+1}
		\end{pmatrix}  \begin{pmatrix}
		\bm 1_{s_2} \\
		\bm F_{:,k+1}
		\end{pmatrix}^{\top} , 
		\end{align*}
		where $\bm H \in \Re^{s_1, s_1}$ and $\bm F \in \Re^{s_2, s_2}$ are Hadamard matrices of order $2^{t_1}$ and order $2^{t_2}$, respectively.
	\end{example}
	In Example~\ref{eg3}, by the construction and the fact that the columns of a Hadamard matrix are orthogonal, the $(k+1)$ left- and right- vectors defining matrix $\bm A$ are orthogonal, respectively. It follows that all the non-zero singular values of $\bm A$ are equal to $\sqrt{mn}$ as $m=2s_1,n=2s_2$, and a pair of top singular vectors can be 
	\begin{align*}
	\bm u_1 = \frac{1}{\sqrt{m}} \begin{pmatrix}
	\bm 1_{s_1} \\
	\bm H_{:,k+1}
	\end{pmatrix}, 
	\bm v_1 = \frac{1}{\sqrt{n}}  \begin{pmatrix}
	\bm 1_{s_2} \\
	\bm F_{:,k+1}
	\end{pmatrix}.
	\end{align*}
	Given  vectors $\bm u_1$ and $\bm v_1$,  Algorithm~\ref{algo:trun3} selects $\hat{S}_1 = [s_1]$ at Step 3 and $\hat{T}_2 = [s_2]$ at Step 5, respectively. 
	
	Then the submatrix $\bm A_{\hat{S}_1, [n]}$ and submatrix $\bm A_{ [m], \hat{T}_2}$ can be written as
	\begin{align*}
	\bm A_{\hat{S}_1, [n]} =
	\begin{pmatrix}
	\bm 1_{s_1, s_2} & \bm 1_{s_1} \bm F_{:,k+1}^{\top}
	\end{pmatrix}, 
	\bm A_{ [m], \hat{T}_2}=
	\begin{pmatrix}
	\bm 1_{s_1, s_2} \\ \bm H_{:,k+1} \bm 1_{s_2}^{\top}
	\end{pmatrix}.
	\end{align*}
	
	In this way, the $s_2$ largest absolute entries in $(\bm u_1)_{\hat{S}_1}^{\top} \bm A_{\hat{S}_1, [n]} $ can be achieved by the $s_2$ first entries so $\hat{S}_2 =[s_2]$. Similarly, Algorithm~\ref{algo:trun3} chooses $\hat{T}_1 =[s_1]$ at Step 6.
	
	For the two solutions $(\hat{S}_1, \hat{S}_2)$ and $(\hat{T}_1, \hat{T}_2)$, both objective values of SSVD \eqref{eq_ssvdcom} are equal to $||\bm A_{[s_1], [s_2]}||_{(k)} $ $= ||\bm 1_{s_1, s_2}||_{(k)} = \sqrt{s_1s_2}$.
	
	However, the true optimal submatrix of SSVD \eqref{eq_ssvdcom} is $\bm A_{[s_1+1, 2s_1], [s_2+1, 2s_2]} $ with the optimal value
	\begin{align*}
	||\bm A_{[s_1+1, 2s_1], [s_2+1, 2s_2]} ||_{(k)} = \bigg|\bigg|\sum_{ i \in [k]} \frac{\sqrt{mn}}{\sqrt{s_1s_2}} \bm H_{:,i} \bm F_{:,i}^{\top} + \bm H_{:,k+1} \bm F_{:,k+1}^{\top} \bigg| \bigg|_{(k)} = k\sqrt{mn}.
	\end{align*}
	Thus, for Example~\ref{eg3}, the approximation ratio of selection Algorithm~\ref{algo:trun3} exactly is $\sqrt{s_1s_2}/(k\sqrt{mn})$.
	
	Finally, the top singular vectors can be obtained using $1$-truncated SVD of matrix $\bm A$, which takes $O(mn)$ according to \cite{shishkin2019fast}. Next, Steps 3 and 4 can be solved by sorting with the time complexity of $m\log(m)$ and $n\log(n)$, respectively. So do Steps 5 and 6. \qed
\end{proof}

It is worthy of mentioning that: (i) the approximation ratios of selection Algorithm~\ref{algo:trun1}, Algorithm~\ref{algo:trun2}, and  Algorithm~\ref{algo:trun3} match the best-known approximation ratio of rank-one SPCA and rank-one SSVD (see, e.g.,  \cite{chan2016approximability,li2020exact});
and (ii) the approximation ratios of the three selection algorithms satisfy Algorithm~\ref{algo:trun1} $\geq$  Algorithm~\ref{algo:trun2} $\geq$   Algorithm~\ref{algo:trun3}, while their scalabilities are in the reverse order. This indicates a clear trade-off between optimality and time complexity when employing them to solve practical problems.

\subsection{Searching Algorithms} \label{subsec:search}
This subsection studies the widely-used greedy and local search algorithms  (see, e.g., \cite{dey2020solving, li2020best}).
We will show that the two searching algorithms have the same approximation ratio for solving SSVD \eqref{eq_ssvdcom}, and the ratio is tight for the cases of $s_1 \le  m/2$ and $s_2 \le n/2$. 

\noindent\textbf{Greedy Algorithm~\ref{algo:svd_greedy}.} The greedy  algorithm for SSVD \eqref{eq_ssvdcom} proceeds as follows: At each iteration, given two subsets denoting the selected rows and columns, respectively, the algorithm first searches for a new row from the unchosen set to maximize the  objective value of the augmented submatrix and adds the optimal row to the row subset, then finds a new column to update the column subset until the two subsets attaining sizes $s_1$ and $s_2$, respectively. 
The implementation details can be found in Algorithm~\ref{algo:svd_greedy}. 

Theorem~\ref{thmg} shows a $1/\sqrt{ks_1s_2}$-approximation ratio and a $O(\max\{s_1,s_2\}(m+n)ks_1s_2)$-time complexity when employing the greedy Algorithm~\ref{algo:svd_greedy} to solve SSVD \eqref{eq_ssvdcom}. 

\begin{algorithm}[htb]
	\caption{Greedy Algorithm for SSVD \eqref{eq_ssvdcom}}
	\label{algo:svd_greedy}
	\begin{algorithmic}[1]
		\State \textbf{Input:} A matrix $\bm A \in \Re^{m \times n}$,  integers   $s_1\in [m]$, $s_2\in [n]$, and  $1\le k \le \min\{s_1, s_2\}$

		\State Compute $(i^*, j^*) \in \argmax_{i \in[m], j \in[n] }\left\{|A_{i, j} |\right\}$
		\State  Define subsets $\hat{S}_1 := \{i^*\}$ and $\hat{S}_2 := \{j^*\}$

		\For{$\ell = 2, \cdots, \max\{s_1,s_2\}$}
		\If{$\ell \le \min\{s_1,s_2\}$}
		
		\State 
		$i^* \in \argmax_{i \in[m]\setminus \hat{S}_1  }
		||\bm A_{\hat{S}_1 \cup\{i\}, \hat{S}_2  }||_{(\min\{\ell, k\})}
		$
		\State Update  $\hat{S}_1 := \hat{S}_1\cup \{i^*\}$ 
		
		\State $j^* \in \argmax_{j \in[n]\setminus \hat{S}_2 }||\bm A_{\hat{S}_1 , \hat{S}_2 \cup\{j\}}||_{(\min\{\ell, k\})}$ 
		\State  Update $\hat{S}_2:= \hat{S}_2 \cup \{ j^*\}$
		
		\ElsIf{$s_1\leq s_2$}
		\State 
		$j^* \in \argmax_{j \in[n]\setminus \hat{S}_2 }||\bm A_{\hat{S}_1 , \hat{S}_2 \cup\{j\}}||_{(k)}$ 
		\State  Update $\hat{S}_2:= \hat{S}_2 \cup \{ j^*\}$
		\Else
		\State 
		$i^* \in \argmax_{i \in[m]\setminus \hat{S}_1  }
		||\bm A_{\hat{S}_1 \cup\{i\}, \hat{S}_2  }||_{(k)}
		$
		\State  Update  $\hat{S}_1 := \hat{S}_1\cup \{i^*\}$
		\EndIf
		\EndFor
		
		\State  \textbf{Output:} $\hat{S}_1 , \hat{S}_2 $
	\end{algorithmic}
\end{algorithm}

\begin{restatable}{theorem}{thmg} \label{thmg} 
	The greedy Algorithm~\ref{algo:svd_greedy} yields a $1/\sqrt{ks_1s_2}$-approximation ratio for SSVD \eqref{eq_ssvdcom}, i.e., suppose that the output of Algorithm~\ref{algo:svd_greedy} is $(\hat{S}_1, \hat{S}_2)$, then
	\begin{align*}
	||\bm A_{\hat{S}_1, \hat{S}_2}||_{(k)} \ge \frac{1}{\sqrt{ks_1s_2}}  w^*,
	\end{align*}
	the ratio is tight {when $s_1\le m/2$ and $s_2\le n/2$}, and the time complexity of the greedy Algorithm~\ref{algo:svd_greedy} is $O(\max\{s_1,s_2\} \\(m+n)ks_1s_2  )$.
\end{restatable}

\begin{proof}
	Suppose that $(S_1^*, S_2^*)$ is an optimal solution to SSVD \eqref{eq_ssvdcom}. According to the last two equalities and inequalities in the proof of Theorem~\ref{thms1}, we have
	\begin{align*}
	w^* = || \bm A_{S_1^*, S_2^*}||_{(k)} \le \sqrt{k}|| \bm A_{S_1^*, S_2^*}||_F =\sqrt{k} \sqrt{\sum_{i\in S_1^*} \sum_{j\in S_2^*} A_{ij}^2}  \le \sqrt{k}\sqrt{s_1s_2} \max_{i\in [m],j\in [n]} |A_{ij}|,
	\end{align*}
	where the first inequality can be found in the proof of Theorem~\ref{thms1},  
	the right-hand side is upper bounded by the output value from greedy Algorithm~\ref{algo:svd_greedy}, and thus we obtain the $1/\sqrt{ks_1s_2}$-approximation ratio. The proof of its tightness is based on the example below.
	
	\begin{example} \label{eg4}
		For any positive integers $k>0, t>0$, suppose that $s_1 = ck$, $s_2 = 2^{t}$ such that $\min\{s_1, s_2\} \ge k$, and $m= 2s_1$, $n=2s_2$. Let us construct matrix $\bm A \in \Re^{m \times n}$ as
		\begin{align*}
		\bm A = \begin{pmatrix}
		1 & \bm 0_{s_2-1}^{\top} & \bm 0_{s_2}^{\top}\\	
		\bm 0_{s_1-1} & \bm 0_{s_1-1, s_2-1} & \bm 0_{s_1-1, s_2}\\
		\bm 0_{s_1} & \bm 0_{s_1, s_2-1}& \bm A_3 \in \Re^{s_1, s_2}\\
		\end{pmatrix}, 
		\end{align*}
		where matrix $\bm A_3$ is defined in Example~\ref{eg2}.
	\end{example}
	
	For Example~\ref{eg4}, the greedy Algorithm~\ref{algo:svd_greedy} starts with $\hat{S}_1 =\{1\}$ and $\hat{S}_2 =\{1\}$ since 1 is the largest absolute entry in matrix $\bm A$. Then at each iteration $\ell$, we update $\hat{S}_1 := \hat{S}_1\cup \{i^*\}$ or $\hat{S}_2 := \hat{S}_2\cup \{j^*\}$ or both depending on the values of $s_1$ and $s_2$. Hence, in Example~\ref{eg4}, the greedy Algorithm~\ref{algo:svd_greedy} outputs the objective value of SSVD \eqref{eq_ssvdcom} as
	\begin{align*}
	||\bm A_{\hat{S}_1, \hat{S}_2}||_{(k)} = ||\bm A_{[s_1], [s_2]}||_{(k)} = 1.
	\end{align*}
	On the other hand, the true optimal value of SSVD \eqref{eq_ssvdcom}  is $||\bm A_3||_{(k)} = \sqrt{ks_1s_2}$. This proves the tightness when $s_1 \le m/2$ and $s_2\le  n/2$.
	
	As for the time complexity, it is known that at most $\max\{s_1, s_2\}$ iterations are required in the for-loop of the greedy Algorithm~\ref{algo:svd_greedy}. At each iteration, solving the two maximization problems is equivalent to computing $k$-truncated SVD of a submatrix of size no larger than $s_1 \times s_2$ at most $(m+n)$ times. According to \cite{shishkin2019fast}, the time complexity of solving $k$-truncated SVD of a size $s_1\times s_2$ matrix is $O(ks_1s_2)$. Hence,  each iteration  takes $O((m+n)ks_1s_2)$ operations. This completes the proof of time complexity. \qed
\end{proof}

\noindent\textbf{Local Search Algorithm~\ref{algo:svd_localsearch}.} The local search algorithm, one of the well-known improved heuristic methods, aims to find a better solution by exploring neighborhoods. Specifically, the local search algorithm for SSVD \eqref{eq_ssvdcom} proceeds as follows: (i) initialize a feasible solution $(\hat{S}_1, \hat{S}_2)$ with the output of the greedy  Algorithm~\ref{algo:svd_greedy}, (ii) at each iteration, first swap an element in $\hat{S}_1$ with an element in $[m]\setminus \hat{S}_1$ and check if such a movement increases the objective value or not- if YES, then accept the swap, and otherwise continue, and (iii) explore the neighborhood solutions by swapping elements between $\hat{S}_2$ and $[n]\setminus\hat{S}_2$. By doing so, the two subsets $\hat{S}_1, \hat{S}_2$ will be updated separately. More details are presented in Algorithm~\ref{algo:svd_localsearch}. Surprisingly, Theorem~\ref{thmls} shows that the local search Algorithm~\ref{algo:svd_localsearch} also yields a tight $1/\sqrt{ks_1s_2}$-approximation ratio. That is, in the worst-case, the local search Algorithm~\ref{algo:svd_localsearch} might not be able to improve the performance of the greedy Algorithm~\ref{algo:svd_greedy} when $s_1\le m/2$ and $s_2\le n/2$, as illustrated in our numerical study.

It is interesting to note that although the theoretical approximation ratios of greedy Algorithm~\ref{algo:svd_greedy} and local search Algorithm~\ref{algo:svd_localsearch}
seem not as good as selection Algorithm~\ref{algo:trun1} and Algorithm~\ref{algo:trun2}, our numerical study shows that the searching algorithms in general work better and can find an optimal solution for most of the testing cases, especially the local search Algorithm~\ref{algo:svd_localsearch}.

\begin{algorithm}[htb]
	\caption{Local Search Algorithm for SSVD \eqref{eq_ssvdcom}}
	\label{algo:svd_localsearch}
	\begin{algorithmic}[1]
		\State \textbf{Input:} A matrix $\bm A \in \Re^{m \times n}$,  integers   $s_1\in [m]$, $s_2\in [n]$, and  $1\le k \le \min\{s_1, s_2\}$
		\State Initialize $(\hat{S}_1 ,\hat{S}_2)$ as the output of greedy Algorithm~\ref{algo:svd_greedy}
		\State \textbf{do}
		\For {each pair {$(i_1,j_1) \in \hat{S}_1 \times ([m]\setminus \hat{S}_1)$}}
		\If{$||\bm A_{\hat{S}_1\cup \{j_1\} \setminus \{i_1\}, \hat{S}_2} ||_{(k)} > || \bm A_{\hat{S}_1,\hat{S}_2}||_{(k)}$}
		\State Update $\hat{S}_1 := \hat{S}_1 \cup \{j_1\} \setminus \{i_1\}$
		\EndIf
		\EndFor
		
		\For{each pair {$(i_2,j_2) \in \hat{S}_2 \times ([n]\setminus\hat{S}_2)$}}
		\If{$||\bm A_{\hat{S}_1, \hat{S}_2 \cup \{j_2\} \setminus \{i_2\}}||_{(k)} > ||\bm A_{\hat{S}_1, \hat{S}_2}||_{(k)}$}
		\State Update  $\hat{S}_2 := \hat{S}_2 \cup \{j_2\} \setminus \{i_2\}$
		\EndIf
		\EndFor
		\State\textbf{while} {there is still an improvement}
		\State \textbf{Output:} $\hat{S}_1$, $\hat{S}_2$
	\end{algorithmic}
\end{algorithm}

\begin{restatable}{theorem}{thmls} \label{thmls} 
	The local search Algorithm~\ref{algo:svd_localsearch} yields a $1/\sqrt{ks_1s_2}$-approximation ratio for SSVD \eqref{eq_ssvdcom}, i.e., suppose that the output of Algorithm~\ref{algo:svd_localsearch} is $(\hat{S}_1, \hat{S}_2)$, then
	\begin{align*}
	||\bm A_{\hat{S}_1, \hat{S}_2}||_{(k)} \ge \frac{1}{\sqrt{ks_1s_2}}  w^*,
	\end{align*}
	and the ratio is tight {when $s_1\le m/2$ and $s_2\le n/2$}.
\end{restatable}

\begin{proof}
	Since local search Algorithm~\ref{algo:svd_localsearch} uses the output solution of greedy Algorithm~\ref{algo:svd_greedy} as the initial solution and local search Algorithm~\ref{algo:svd_localsearch} is an improved heuristic method, the approximation ratio of Algorithm~\ref{algo:svd_localsearch} must be no worse than that of greedy Algorithm~\ref{algo:svd_greedy}, i.e., $1/\sqrt{ks_1s_2}$.
	
	Moreover, we show that the $1/\sqrt{ks_1s_2}$-approximation ratio is also tight for local search Algorithm~\ref{algo:svd_localsearch} and the ratio can be achieved by Example~\ref{eg4}. Indeed, in Example~\ref{eg4}, swapping one element from the chosen set and one element from the unchosen set, the resulting objective value is still equal to 1 and does not improve the current solution. It follows that the output of local search Algorithm~\ref{algo:svd_localsearch} is the same as that of Algorithm~\ref{algo:svd_greedy} with the objective value  1, which yields $1/\sqrt{ks_1s_2}$-approximation ratio for SSVD \eqref{eq_ssvdcom}. \qed
\end{proof}


For the Steps 5 and 10 of the local search Algorithm~\ref{algo:svd_localsearch}, we multiply the right-hand side of the strict inequalities by a factor $(1+\delta)$, where $\delta>0$ denotes a strict improvement factor. Then the local search Algorithm~\ref{algo:svd_localsearch} possesses polynomial time complexity. Note that since  the local search Algorithm~\ref{algo:svd_localsearch} admits the same approximation ratio as the greedy Algorithm~\ref{algo:svd_greedy}, the extra factor $(1+\delta)$ does not affect the approximation ratio of the local search Algorithm~\ref{algo:svd_localsearch}.
\begin{proposition}\label{ls_time}
	Time complexity of the revised local search Algorithm~\ref{algo:svd_localsearch} is 
	$O((nks_1^2s_2+mks_1s_2^2)L/\delta)$ where $\delta >0$ denotes the strict improvement factor and $L:=\lceil\log_2   \sqrt{ks_1s_2}\rceil + \ell_1+ \ell_2$ with $\ell_1$ denoting the maximum binary encoding length of the absolute value of each entry in matrix $\bm A$ and $\ell_2$ denoting the encoding length of the smallest fraction.
\end{proposition}
\begin{proof}
	First, according to the proof of Theorem~\ref{thmg}, we can derive that for any submatrix $\bm A_{\hat{S}_1, \hat{S}_2}$, the following always holds
	\begin{align*}
	||\bm A_{\hat{S}_1, \hat{S}_2}||_{(k)} \le \sqrt{ks_1s_2} \max_{i\in [m], j\in [n]} |\bm A_{ij}|. 
	\end{align*}
	As $\ell_1$ is the maximum binary encoding length of $|\bm A_{ij}|$ for each $i\in [m], j\in [n]$, $||\bm A_{\hat{S}_1, \hat{S}_2}||_{(k)}$ can be further upper bounded by
	\begin{align*}
	||\bm A_{\hat{S}_1, \hat{S}_2}||_{(k)} \le 2^{\ell_1}\sqrt{ks_1s_2}.
	\end{align*}
	According to the definition of $\ell_2$, 
	the largest and smallest value of $||\bm A_{\hat{S}_1, \hat{S}_2}||_{(k)} $ can be $2^{\ell_1}\sqrt{ks_1s_2}$ and $2^{-\ell_2}$. Suppose that we the accept the swapping if the improvement ratio is at least $1+\delta$, the number of swapping steps of local search Algorithm~\ref{algo:svd_localsearch} is at most
	\[\log_{1+\delta} (2^{\ell_1+\ell_2}\sqrt{ks_1s_2} ) \le \frac{2}{\delta}L,\]
	where $L:=\lceil\log_2   \sqrt{ks_1s_2}\rceil + \ell_1+ \ell_2$.		
	
	Finally, for each swapping at Steps 4-13, there are at most $(ns_1+ms_2)$ distinct evaluations of the $k$-truncated SVD of a size $s_1\times s_2$ submatrix, which takes at most $O(nks_1^2s_2+mks_1s_2^2)$. Therefore, local search Algorithm~\ref{algo:svd_localsearch} has an overall time complexity $O((nks_1^2s_2+mks_1s_2^2)L/\delta)$. \qed
\end{proof}

	\section{Extensions to Row-Sparse PCA} \label{sec:spca}
The general rank-$k$ SPCA \eqref{eq_spca} is closely related to SSVD \eqref{eq_ssvdcom}, which has been recently studied in \cite{dey2020solving, vu2012minimax}. The key ingredient of SPCA \eqref{eq_spca} is that its data matrix $\bm{A}\in \S_+^n$ is symmetric and positive semidefinite. 
We show that SPCA \eqref{eq_spca} can be viewed as a special case of SSVD \eqref{ssvd}. It should be noted that matrix $\bm A$ is assumed to be positive semidefinite for SPCA but can be non-symmetric for SSVD \eqref{ssvd}. 

\begin{proposition}\label{prop2}
	When matrix $\bm A $ is positive semidefinite and $s_1 = s_2=s$, SSVD \eqref{ssvd} reduces to SPCA \eqref{eq_spca}, i.e., there exists an optimal solution $(\bm U^*, \bm V^*)$ to SSVD \eqref{ssvd} satisfying $\bm U^* = \bm V^*$.
\end{proposition}
\begin{proof}
	When $\bm A \in  \S_+^n $  and $s_1 = s_2=s$,  SSVD \eqref{ssvd} becomes 
	\begin{align}\label{eq_ssvd2}
	\max_{\bm U \in \Re^{n \times k}, \bm V \in \Re^{n \times k}}  \left\{\tr(\bm U^{\top}\bm A \bm V) :   \bm U^{\top} \bm U =\bm V^{\top} \bm V= \bm I_k, ||\bm U||_0 \le s,   ||\bm V||_0 \le s \right\}.
	\end{align}
	Then for any feasible solution $(\bm U,\bm V)$ to problem \eqref{eq_ssvd2}, suppose that for each $i\in [k]$, $\bm u_i$ and $\bm v_i$ are the $i$th columns of $\bm U$ and $\bm V$, respectively,  then the objective value satisfies
	\begin{align*}
	&\tr(\bm U^{\top}\bm A \bm V) = \sum_{i\in [k]} \bm u_i^{\top}\bm A \bm v_i = \sum_{i\in [k]} \bm u_i^{\top}\bm A^{1/2} \bm A^{1/2}  \bm v_i  \le \sum_{i\in [k]} \sqrt{\bm u_i^{\top}\bm A \bm u_i \bm v_i^{\top} \bm A \bm v_i} \\
	&\le \sum_{i\in [k]} \frac{1}{2} \left({\bm u_i^{\top}\bm A \bm u_i + \bm v_i^{\top} \bm A \bm v_i}\right) = \frac{1}{2}\left(\tr(\bm U^{\top} \bm A \bm U) + \tr(\bm V^{\top} \bm A \bm V)  \right) \le \max \left\{ \tr(\bm U^{\top} \bm A \bm U), \tr(\bm V^{\top} \bm A \bm V)  \right\}.
	\end{align*}
	The above analysis states that for any feasible solution $(\bm U,\bm V)$ to problem \eqref{eq_ssvd2}, there exists a better solution between $(\bm U,\bm U)$ and $(\bm V,\bm V)$. This implies that there must  exist $\bm U^* = \bm V^*$ at optimality. \qed  
\end{proof}

Using Proposition~\ref{prop2} and the fact that  the rank-one SPCA has been proven to be NP-hard and  inapproximable  in \cite{magdon2017np}, we prove the following complexity result for SSVD \eqref{ssvd}.
\begin{corollary} \label{cor1_com}
	The	SSVD \eqref{ssvd} is NP-hard to solve and is inapproximable by any polynomial-time algorithm with a constant-factor approximation ratio.
\end{corollary}

Following SSVD \eqref{eq_ssvdcom}, we also derive an equivalent combinatorial formulation for SPCA \eqref{eq_spca}, i.e.,
\begin{align} \label{eq_spcacom}
\text{(SPCA)} \quad w^{spca}:=	\max_{\begin{subarray}{c}
	S\subseteq [n],  |S|\le s
	\end{subarray}
} ||\bm A_{S, S}||_{(k)}.
\end{align}
%

\subsection{Exact Formulations of SPCA}
Following the spirit of Theorem~\ref{thm_sdp}, we can also derive two exact MISDP formulations for SPCA \eqref{eq_spca}. Since Theorem~\ref{thm_sdp} replies on Lemma~\ref{lem:svd} that reformulates the Ky Fan $k$-norm of a matrix as an SDP, we next show the similar result in Lemma~\ref{lem:eigen} for the positive semidefinite matrix.

\begin{lemma} \label{lem:eigen}
	For any matrix $\bm B \in \S_+^n$ with rank $r$ and an integer $k \in [r]$,  let $\bm B = \sum_{i\in [r]} \lambda_i(\bm B)\bm q_i \bm q_i^{\top}$ denote the eigen-decomposition of matrix $\bm B$. Then 
	we have
	\begin{align*}
	\max_{\bm X \in \S_+^n}  \left\{\tr(\bm B \bm X): \bm X \preceq \bm I_n, \tr(\bm X) \le k  \right\} = ||\bm B||_{(k)},
	\end{align*}
	where the equality can be achieved by $\bm X^* = \sum_{ i \in [k]}\bm q_i \bm q_i^{\top}$.
\end{lemma}
\begin{proof}
	According to Lemma~\ref{lem:svd}, $||\bm B||_{(k)}$ can be formulated by
	\begin{align}\label{eq_temp}
	\max_{\bm X \in \Re^{n \times n} }  \left\{\tr(\bm B \bm X): ||\bm X||_2 \le 1, ||\bm X||_{*}\le k  \right\},
	\end{align}
	which  relaxes the maximization problem in Lemma~\ref{lem:eigen}.
	
	Given that the optimal value of  problem \eqref{eq_temp} is equal to $||\bm B||_{(k)}$, the optimal value of the maximization problem in Lemma~\ref{lem:eigen} must be upper bounded by $||\bm B||_{(k)}$. On the other hand, the upper bound $||\bm B||_{(k)}$ is achievable by the feasible solution $\bm X^* = \sum_{ i \in [k]}\bm q_i \bm q_i^{\top}$ since
	\begin{align*}
	\tr(\bm B \bm X^*) = \sum_{ i \in [k]} \bm q_i^{\top} \bm B\bm q_i = \sum_{ i \in [k]}\lambda_i(\bm B) = ||\bm B||_{(k)},
	\end{align*}
	where the last equality is because for any positive semidefinite matrix, its singular values are equal to the eigenvalues. This completes the proof. \qed
\end{proof}

We are ready to derive the first exact MISDP for SPCA \eqref{eq_spca}.
\begin{corollary} \label{cor1}
	The SPCA \eqref{eq_spca} is equivalent to
	\begin{align} \label{eq_spcasdp1}
	\text{\rm (SPCA)}  \ \ w^{spca} :=	\max_{\bm z \in \{0,1\}^n, \bm X \in \S_+^n} \bigg\{\tr(\bm A \bm X):  \bm X \preceq \bm I_n, \tr(\bm X) \le k,  X_{ii} \le z_i, \forall i\in [n], \sum_{i\in [n]} z_i \le s   \bigg\}.
	\end{align}
\end{corollary}
\begin{proof}
	The result directly follows from the proof of Theorem~\ref{thm_sdp} and the identity in Lemma~\ref{lem:eigen}. Thus, the proof is omitted for brevity. \qed
\end{proof}

For matrix $\bm A \in \S_+^n$ with rank $r$, let $\bm A =\bm C^{\top} \bm C$ denote the Cholesky decomposition, where $\bm C \in \Re^{r \times n}$ is the corresponding Cholesky matrix and for each $i\in [n]$ with $\bm c_i\in \Re^r$ denoting its $i$th column.  Then, we obtain another exact MISDP for SPCA \eqref{eq_spca}. 
\begin{corollary} \label{cor2}
	The SPCA \eqref{eq_spca} is equivalent to
	\begin{align}\label{eq_spcasdp2}
	\text{\rm (SPCA)}  \ \ w^{spca} :=	\max_{ \begin{subarray}{c} \bm z \in \{0,1\}^n, \bm X \in \S_+^r, \\
		\bm W_1, \cdots, \bm W_n \in \S_+^r
		\end{subarray}
	} \bigg\{\sum_{i\in [n]}\bm c_i^{\top} \bm W_i \bm c_i:  &\bm X \preceq \bm I_r, \tr(\bm X) \le k, \bm X \succeq \bm W_i,  \forall i\in [n],  \nonumber  \\
	&	\tr(\bm W_i) \le z_i, \forall i\in [n],  \sum_{i\in [n]} z_i \le s   \bigg\}.
	\end{align}
\end{corollary}
\begin{proof}
	Given $\bm A = \bm C^{\top}\bm C$, for any feasible subset $S$ to SPCA \eqref{eq_spcacom}, its objective value is equal to
	\begin{align*}
	||\bm A_{S, S}||_{(k)} &=\left\|\bm C_{[r], S}^{\top} \bm C_{[r], S}\right\|_{(k)}=\left\|\bm C_{[r], S}\bm C_{[r], S}^{\top} \right\|_{(k)}  
	= \bigg|\bigg|\sum_{ i \in S} \bm c_i \bm c_i^{\top}\bigg|\bigg|_{(k)},
	\end{align*}
	where the second equality is because for any matrix $\bm B$,  $\bm B^{\top}\bm B$ and $\bm B\bm B^{\top}$ are positive semidefinite and have the same non-zero eigenvalues.
	
	Therefore, SPCA \eqref{eq_spcacom} can be formulated as
	\begin{align*}
	\max_{S\subseteq [n], |S|\le s}  \bigg|\bigg|\sum_{ i \in S} \bm c_i \bm c_i^{\top}\bigg|\bigg|_{(k)}.
	\end{align*}
	Introducing binary variables $\bm z\in \Re^n$ such that for all $i\in[n]$, let $z_i=1$ if $i\in S$ and zero; otherwise, the above formulation can be written as
	\begin{align*}
	\max_{\bm z\in \Re^n} \bigg\{  \bigg|\bigg|\sum_{ i \in [n]} z_i \bm c_i \bm c_i^{\top} \bigg|\bigg|_{(k)}: \sum_{ i \in [n]} z_i \le s \bigg\}.
	\end{align*}
	Using the result in Lemma~\ref{lem:eigen} and plugging the SDP formulation of the objective function above, the SPCA now is equivalent to
	\begin{align*}
	\max_{\bm z\in \Re^n, \bm X \in \S_+^r} \bigg\{ \tr\bigg(\sum_{ i \in [n]}  z_i \bm X \bm c_i \bm c_i^{\top}\bigg): \bm X \preceq \bm I_r, \tr(\bm X) \le k,  \sum_{ i \in [n]} z_i \le s \bigg\}.
	\end{align*}
	Following the proof of theorem 1 in \cite{li2020exact}, the bilinear terms $\{ z_i \bm X\}_{i\in [n]}$ can be linearized and finally, the SPCA  \eqref{eq_spca}  can be formulated by MISDP \eqref{eq_spcasdp2}.
	\qed	
\end{proof}

\subsection{Selection Algorithms of SPCA}
As a special case of SSVD \eqref{eq_ssvdcom}, it is natural to extend the proposed approximation algorithms of SSVD \eqref{eq_ssvdcom} to solve SPCA \eqref{eq_spcacom}.
In this subsection, we tailor the selection algorithms to solve SPCA \eqref{eq_spcacom}. In particular, the extensions of selection algorithms are built on the following key lemma.
\begin{lemma} \label{lem:kyfan}
	For any matrix $\bm B \in \S_+^n$ and two equal-sized subsets $S_1, S_2\subseteq [n]$, we have
	\begin{align*}
	||\bm B_{S_1, S_2}||_{(k)} \le  \max \{||\bm B_{S_1, S_1}||_{(k)}, ||\bm B_{S_2, S_2}||_{(k)}  \},
	\end{align*}
	where the positive integer $k \le |S_1|$.
\end{lemma}
\begin{proof}
	When $\bm B \in \S_+^n$ and $|S_1|=|S_2|=s$,	for any optimal solution $(\bm U^*, \bm V^*)$ to  problem \eqref{eq_ssvd2} with $\bm A :=\bm B$, according to the inequality in the proof of Proposition~\ref{prop2}, we show
	\begin{align*}
	||\bm B_{S_1, S_2}||_{(k)} &= \tr((\bm U^*)^{\top} \bm B \bm V^*) \le \max \{\tr((\bm U^*)^{\top} \bm B \bm U^*), \tr((\bm V^*)^{\top} \bm B \bm V^*) \} \\
	&\le  \max \{||\bm B_{S_1, S_1}||_{(k)}, ||\bm B_{S_2, S_2}||_{(k)}  \},
	\end{align*}
	where the last inequality is because matrices $\bm U^*$  and $ \bm V^*$ are feasible to problems 
	\begin{align*}
	\max_{\begin{subarray}{c}
		\bm U\in \Re^{n \times k}
		\end{subarray}}  \left\{\tr(\bm U^{\top} \bm B \bm U): \bm U^{\top} \bm U =\bm I_{k}, \bm U_{[n]\setminus {S}_1, [k]}=\bm 0_{n-s, k} \right\} = ||\bm B_{S_1, S_1}||_{(k)},\\
	\max_{\begin{subarray}{c}
		\bm V\in \Re^{n \times k}
		\end{subarray}}  \left\{\tr(\bm V^{\top} \bm B \bm V): \bm V^{\top} \bm V =\bm I_{k}, \bm V_{[n]\setminus {S}_2, [k]}=\bm 0_{n-s, k} \right\} = ||\bm B_{S_2, S_2}||_{(k)},
	\end{align*} 
	where the equalities are due to theorem 23.2 in \cite{shalev2014understanding}. \qed
\end{proof}

In the selection Algorithm~\ref{algo:trun1}-Algorithm~\ref{algo:trun3} for SSVD \eqref{eq_ssvdcom}, the output is a pair of subsets $(\hat{S}_1, \hat{S}_2)$, while SPCA \eqref{eq_spcacom} only involves one subset. Given $\bm A\in \S_+^n$ and $s_1=s_1=s$, according to Lemma~\ref{lem:kyfan}, we can adapt Algorithm~\ref{algo:trun1}-Algorithm~\ref{algo:trun3} to solve SPCA \eqref{eq_spcacom} by comparing the objective values of $\bm A_{\hat{S}_1, \hat{S}_1}$ and $\bm A_{\hat{S}_2, \hat{S}_2}$ and outputting a better solution of $\hat{S}_1$ and $ \hat{S}_2$. 
In particular, since matrix $\bm A$ is symmetric and positive semidefinite in SPCA \eqref{eq_spcacom},
we can simplify the implementation of the selection Algorithm~\ref{algo:trun2} by using one loop on rows only;  and the third selection Algorithm~\ref{algo:trun3} can also be simplified to one loop only for solving SPCA \eqref{eq_spcacom} since the top left- and right- singular vectors of positive semidefinite matrices are the same. The details of the selection algorithms for SPCA \eqref{eq_spcacom} can be found in  Algorithm~\ref{algo:trun1spca},  Algorithm~\ref{algo:trun2spca}, and  Algorithm~\ref{algo:trun3spca}. 
Their corresponding approximation ratios and time complexities for   SPCA \eqref{eq_spcacom} are summarized in Corollary~\ref{cor_ratio}.

\begin{algorithm}[htb]
	\caption{Selection Algorithm for SPCA \eqref{eq_spcacom}}
	\label{algo:trun1spca}
	\begin{algorithmic}[1]
		\State  {\bfseries Input:} A matrix $\bm A \in \S_+^{ n}$ and  integers   $s_1 =  s_2 = s\in [n]$
		\State  Suppose $(\hat{S}_1, \hat{S}_2)$ solves the problem below
		\begin{align*}
		\ \max_{
			\begin{subarray}{c}
			S_1\subseteq[m],  S_2\subseteq [n]
			\end{subarray}
		} \left\{||\bm A_{S_1, S_2}||_F: |S_1| \le s_1,|S_2| \le s_2\right\}.
		\end{align*}
		\State  {\bfseries Output:} A better solution of SPCA \eqref{eq_spcacom} between $\hat{S}_1$ and $\hat{S}_2$
	\end{algorithmic}
\end{algorithm}

\begin{algorithm}[htb]
	\caption{Selection Algorithm for SPCA \eqref{eq_spcacom}}
	\label{algo:trun2spca}
	\begin{algorithmic}[1]
		\State {\bfseries Input:} A matrix $\bm A \in \S_+^n$,  integers   $s_1=s_2=s \in [n]$, and  $1\le k \le s$
		\For{$j\in [n]$}
		\State  Let $\hat{S}_1^j $ denote the index set of the $s_1$ largest absolute entries in $\bm A_{:,j}$
		\State  
		$
		\hat{S}_2^j \in \argmax_{S\subseteq [n]} \{\sum_{\ell \in S} ||\bm A_{\hat{S}_1^j, {\ell} }||_2: |S|=s_2 \}
		$		
		\EndFor
		
		\State Let $(\hat{S}_1, \hat{S}_2)$ denote the best solution for SSVD \eqref{eq_ssvdcom} among the candidate set $\{(\hat{S}_1^j, \hat{S}_2^j)\}_{j\in [n]} $
		\State  {\bfseries Output:} A better solution of SPCA \eqref{eq_spcacom} between $\hat{S}_1$ and $\hat{S}_2$
	\end{algorithmic}
\end{algorithm}

\begin{algorithm}[htb] 
	\caption{Selection Algorithm for SPCA \eqref{eq_spcacom}} \label{algo:trun3spca}
	\begin{algorithmic}[1]
		\State \textbf{Input:} A matrix $\bm A \in \S_+^n$,  integers   $s_1=s_2=s \in [n]$, and  $1\le k \le s$
		
		\State Compute the eigenvector  $\bm u_1$ corresponding to the largest eigenvalue of  $\bm A$
		
		\State Denote by $ \hat{S}_1$ the indices of $s_1$ largest absolute entries in $\bm u_1$
		\State Denote by $ \hat{S}_2$ the indices of $s_2$ largest absolute entries in $(\bm u_1)_{\hat{S}_1}^{\top} \bm A_{\hat{S}_1, [n]} $
		
		
		
		\State	\textbf{Output:} A better solution of SPCA \eqref{eq_spcacom} between $\hat{S}_1$ and $\hat{S}_2$
	\end{algorithmic}
\end{algorithm}

\begin{restatable}{corollary}{} \label{cor_ratio}
	The selection Algorithm~\ref{algo:trun1spca}-Algorithm~\ref{algo:trun3spca} for  SPCA \eqref{eq_spcacom} admit $1/\sqrt{s}$, $1/\sqrt{ks}$, and $s/(kn)$ approximation ratios, respectively, and all the ratios are tight {when $s\le n/2$}; Moreover, the selection Algorithm~\ref{algo:trun2spca} and Algorithm~\ref{algo:trun3spca}  admit $O(n^2log(n)+nks^2)$ and $O(n\log(n)+n^2)$ time complexities,  respectively.
\end{restatable}
\begin{proof}
	The proof includes two parts with respect to the approximation ratios and time complexities. Plugging $m=n$ and $s_1=s_2=s$ into the complexities of selection Algorithm~\ref{algo:trun2} and Algorithm~\ref{algo:trun3} for SSVD \eqref{eq_ssvdcom}, we obtain the desired time complexities of Algorithm~\ref{algo:trun2spca} and Algorithm~\ref{algo:trun3spca}.
	
	In the following, we will  focus on deriving the approximation ratios of the selection algorithms for  SPCA \eqref{eq_spcacom}, separately.  When matrix $\bm A \in \S_+^n$ and $s_1=s_2=s$, we have the following results.
	\begin{enumerate}[(i)]
		\item 	\textbf{Selection Algorithm~\ref{algo:trun1spca} for  SPCA \eqref{eq_spcacom}.}
		
		Suppose that the output of selection Algorithm~\ref{algo:trun1spca}  is $(\hat{S}_1, \hat{S}_2)$. 
		\begin{align*}
		w^{spca} = w^* \le \frac{1}{\sqrt{\min\{s_1, s_2\}}} ||\bm A_{\hat{S}_1, \hat{S}_2}||_{(k)}  \le \frac{1}{\sqrt{s}} \max\left\{||\bm A_{\hat{S}_1, \hat{S}_1}||_{(k)}, ||\bm A_{\hat{S}_2, \hat{S}_2}||_{(k)} \right\},
		\end{align*}
		where the first inequality is due to Theorem~\ref{thms1} and the second one is from Lemma~\ref{lem:kyfan}.
		
		Thus the modified selection algorithm yields the $1/\sqrt{s}$-approximation ratio for  SPCA \eqref{eq_spcacom}. The ratio is tight since  the following example can achieve it.
		\begin{example} \label{eg5}
			Suppose that $k=s$, $n =2s$, and matrix $\bm A \in \S_+^n$ is defined by
			\begin{align*}
			\bm A = \begin{pmatrix}
			\bm 1_{s, s} & \bm 0_{s, s}\\
			\bm 0_{s, s} & \bm D
			\end{pmatrix}, \bm D \in \S_+^s:= \sum_{i\in [k]} \sqrt{s} \bm e_i \bm e_i^{\top},
			\end{align*}
			where $\bm e_i \in \Re^{s}$ is the $i$th column of the identity matrix $\bm I_{s}$.
		\end{example}
		In Example~\ref{eg5},	the selection Algorithm~\ref{algo:trun1spca} returns a solution $\hat{S}_1 = [s]$ and $\hat{S}_2 = [s]$ since $||\bm 1_{s, s}||_F= ||\bm D||_F $, and the resulting objective value of  SPCA \eqref{eq_spcacom}  is $$\max\left\{||\bm A_{\hat{S}_1, \hat{S}_1}||_{(k)}, ||\bm A_{\hat{S}_2, \hat{S}_2}||_{(k)} \right\} = ||\bm 1_{s, s}||_{(k)}=s.$$
		However, the true optimal value of SPCA \eqref{eq_spcacom} is $||\bm D||_{(k)}=k\sqrt{s} = s\sqrt{s}$, which thus shows that the $1/\sqrt{s}$-ratio is tight.
		
		\item 	\textbf{Selection Algorithm~\ref{algo:trun2spca} for  SPCA \eqref{eq_spcacom}.}
		
		Suppose that the output of selection Algorithm~\ref{algo:trun2spca}  is $(\hat{S}_1, \hat{S}_2)$, similarly, we have
		\begin{align*}
		w^{spca} = w^* \le \frac{1}{\sqrt{k\min\{s_1, s_2\}}} ||\bm A_{\hat{S}_1, \hat{S}_2}||_{(k)}  \le \frac{1}{\sqrt{ks}} \max\left\{||\bm A_{\hat{S}_1, \hat{S}_1}||_{(k)}, ||\bm A_{\hat{S}_2, \hat{S}_2}||_{(k)} \right\}.
		\end{align*}
		
		The ratio is tight since  the following example can achieve it.
		\begin{example} \label{eg6}
			Suppose that $k=s$, $n =2s$, and matrix $\bm A \in \S_+^n$ is defined by
			\begin{align*}
			\bm A = \bm e_1 \bm e_1^{\top} + \sum_{ i \in [s+1, n]} \bm e_i \bm e_i^{\top},
			\end{align*}
			where for each $i\in [n]$, let $\bm e_i \in \Re^{n}$ denote the $i$th column of the identity matrix $\bm I_{n}$.
		\end{example}
		In Example~\ref{eg6},	the selection Algorithm~\ref{algo:trun2spca} proceeds as follows: (a) for each $j\in [s]$, we can choose $\hat{S}_1^j = \hat{S}_2^j = [s]$ and the resulting objective value is equal to 1; (b) for each $j\in [s+1, n]$, we can choose $\hat{S}_1^j = \hat{S}_2^j = [2,s]\cup\{j\}$ and the resulting objective value is also equal to 1. In this way, the final objective value of  SPCA \eqref{eq_spcacom}  is equal to 1.
		However, the true optimal value of SPCA \eqref{eq_spcacom} is $||\bm  A_{[s+1, n], [s+1, n]}||_{(k)}=k = s$, which thus shows that the $1/\sqrt{ks}$-ratio is tight when $s\le n/2$.
		
		\item 	\textbf{Selection Algorithm~\ref{algo:trun3spca} for  SPCA \eqref{eq_spcacom}.}
		
		Suppose that the output of selection Algorithm~\ref{algo:trun3spca}  is $(\hat{S}_1, \hat{S}_2)$, then we have
		\begin{align*}
		w^{spca} = w^* \le \frac{\sqrt{s_1s_2}}{k \sqrt{mn}} ||\bm A_{\hat{S}_1, \hat{S}_2}||_{(k)}  \le \frac{s}{kn} \max\left\{||\bm A_{\hat{S}_1, \hat{S}_1}||_{(k)}, ||\bm A_{\hat{S}_2, \hat{S}_2}||_{(k)} \right\}.
		\end{align*}
		
		The ratio is tight since  the following example can achieve it.
		\begin{example} \label{eg7}
			For any positive integer $k>0, t>0$,	suppose that $s=2^t$ satisfying $s\ge k+1$. Let us define $n =2s$ and matrix $\bm A \in \S_+^n$ as
			\begin{align*}
			\bm A =  \sum_{i \in [k]}  \frac{n}{s} \begin{pmatrix}
			\bm 0_{s} \\
			\bm H_{:,i}
			\end{pmatrix}  \begin{pmatrix}
			\bm 0_{s} \\
			\bm H_{:,i}
			\end{pmatrix}^{\top} + \begin{pmatrix}
			\bm 1_{s} \\
			\bm H_{:,k+1}
			\end{pmatrix}  \begin{pmatrix}
			\bm 1_{s} \\
			\bm H_{:,k+1}
			\end{pmatrix}^{\top} , 
			\end{align*}
			where $\bm H \in \Re^{s, s}$ is a Hadamard matrix of order $2^{t}$.
		\end{example}
		In Example~\ref{eg7}, in the construction of matrix $\bm A$, the $(k+1)$ vectors are  orthogonal because the columns of a Hadamard matrix are orthogonal. It follows that all the non-zero eigenvalues of $\bm A$ are equal to $n$ since $n=2s$. Thus, a top eigenvector can be 
		\begin{align*}
		\bm u_1 = \frac{1}{\sqrt{n}} \begin{pmatrix}
		\bm 1_{s} \\
		\bm H_{:,k+1}
		\end{pmatrix}.
		\end{align*}
		Given  vector $\bm u_1$,  Algorithm~\ref{algo:trun3spca} selects $\hat{S}_1 = [s]$ at Step 3 and $\hat{S}_2 = [s]$ at Step 4, respectively. Since $\hat{S}_1 = \hat{S}_2 = [s]$, the objective value of SPCA \eqref{eq_spcacom}  is equal to $||\bm A_{[s], [s]}||_{(k)} = ||\bm 1_{s, s}||_{(k)} =s$. 
		
		However, the true optimal submatrix of SSVD \eqref{eq_ssvdcom} is $\bm A_{[s+1, 2s], [s+1, 2s]} $ and the resulting optimal value is equal to
		\begin{align*}
		||\bm A_{[s+1, 2s], [s+1, 2s]} ||_{(k)} = \bigg|\bigg|\sum_{ i \in [k]} \frac{n}{s} \bm H_{:,i} \bm H_{:,i}^{\top} + \bm H_{:,k+1} \bm H_{:,k+1}^{\top} \bigg| \bigg|_{(k)} = kn,
		\end{align*}
		which confirms that the $s/(kn)$-approximation ratio is tight. \qed
	\end{enumerate}
\end{proof}
%

\subsection{Searching Algorithms of SPCA}
The positive semidefiniteness of matrix $\bm A$ allows us to simplify the implementations and improves theoretical performances of the greedy and local search algorithms for SPCA \eqref{eq_spcacom} compared to SSVD \eqref{eq_ssvdcom}. The detailed procedures are presented in Algorithm~\ref{algo:spca_greedy} and Algorithm~\ref{algo:spca_localsearch}. It is worthy of mentioning that at the $\ell$th iteration of  the greedy Algorithm~\ref{algo:spca_greedy},  if $\ell\leq k$, Step 4  reduces to
\[i^* \in \argmax_{i \in[n]\setminus \hat{S}  }
\tr(\bm A_{\hat{S} \cup\{i\}, \hat{S}\cup \{i\}  }) =\argmax_{i \in[n]\setminus \hat{S}  }
A_{ii} .\]
Thus, at the first $k$ iterations, the greedy Algorithm~\ref{algo:spca_greedy} selects the $k$ largest diagonal elements in matrix $\bm A$. 

\begin{algorithm}[h]
	\caption{Greedy Algorithm for SPCA \eqref{eq_spcacom}}
	\label{algo:spca_greedy}
	\begin{algorithmic}[1]
		\State \textbf{Input:} A matrix $\bm A \in \S_+^n$,  integers   $s \in [n]$, and  $1\le k \le s$
		
		\State  Define the subset $\hat{S} := \emptyset$ 
		
		\For{$\ell \in [s]$}
		
		\State 
		$i^* \in \argmax_{i \in[n]\setminus \hat{S}  }
		||\bm A_{\hat{S} \cup\{i\}, \hat{S}\cup \{i\}  }||_{(\min\{\ell, k\})}
		$
		\State Update  $\hat{S} := \hat{S} \cup \{i^*\}$ 
		
		\EndFor
		
		\State  \textbf{Output:} $\hat{S}  $
	\end{algorithmic}
\end{algorithm}

\begin{algorithm}[h]
	\caption{Local Search Algorithm for SPCA \eqref{eq_spcacom}}
	\label{algo:spca_localsearch}
	\begin{algorithmic}[1]
		\State \textbf{Input:} A matrix $\bm A \in \S_+^n$,  integers   $s \in [n]$, and  $1\le k \le s$
		\State Initialize $\hat{S}$ as the output of greedy Algorithm~\ref{algo:spca_greedy}
		\State \textbf{do}
		\For {each pair {$(i,j) \in \hat{S} \times ([n]\setminus \hat{S})$}}
		\If{$||\bm A_{\hat{S}\cup \{j\} \setminus \{i\}, \hat{S}\cup \{j\} \setminus \{i\}} ||_{(k)} > || \bm A_{\hat{S},\hat{S}}||_{(k)}$}
		\State Update $\hat{S} := \hat{S} \cup \{j\} \setminus \{i\}$
		\EndIf
		\EndFor
		
		\State\textbf{while} {there is still an improvement}
		\State \textbf{Output:} $\hat{S}$
	\end{algorithmic}
\end{algorithm}
We also derive the approximation ratios and time complexities of greedy Algorithm~\ref{algo:spca_greedy} and local search Algorithm~\ref{algo:spca_localsearch} for SPCA.
\begin{corollary}
	The greedy Algorithm~\ref{algo:spca_greedy} and local search  Algorithm~\ref{algo:spca_localsearch} for solving SPCA \eqref{eq_spcacom} admit a $k/s$-approximation ratio and the ratio is tight for both algorithms {when $s\le n/2$}. Moreover, the greedy Algorithm~\ref{algo:spca_greedy} and local search  Algorithm~\ref{algo:spca_localsearch} have time complexities of $O(nks^3)$ and $O( nks^3L/\delta)$, respectively.
\end{corollary}
\begin{proof}
	Let us first analyze the approximation ratio of greedy Algorithm~\ref{algo:spca_greedy}.
	To begin with, let $T$ denote the index set of $k$ largest diagonal elements in matrix $\bm A$.
	Suppose that an optimal solution to SPCA \eqref{eq_spcacom} is set $S^*$ and the output of greedy Algorithm~\ref{algo:spca_greedy} is $\hat{S}$, then we have
	\begin{align*}
	w^{spca} = ||\bm A_{S^*, S^*}||_{(k)} = \sum_{ i \in [k]} \lambda_i(\bm A_{S^*, S^*}) \le \tr (\bm A_{S^*, S^*}) = \sum_{ i \in S^*} A_{ii} \le \frac{k}{s} \sum_{ i \in T} A_{ii} \le \frac{k}{s} ||\bm A_{\hat{S}, \hat{S}}||_{(k)}, 
	\end{align*}
	where the last inequality is because at the first $k$ iterations, greedy Algorithm~\ref{algo:spca_greedy} picks the $k$ largest diagonal elements of matrix $\bm A$. Thus, the greedy Algorithm~\ref{algo:spca_greedy} yields the $k/s$-approximation ratio for  SPCA \eqref{eq_spcacom}. The ratio is tight since  the following example can achieve it.
	\begin{example}\label{eg8}
		For any positive integers $s\geq k>0$, let $n=2s$ and matrix $\bm A \in \S_+^n$ be
		\begin{align*}
		\bm A = \sum_{ i \in [s]} \bm e_i \bm e_i^{\top} + \sum_{ i \in [k]} \begin{pmatrix}
		\bm 0_s \\
		\bm 1_{s}
		\end{pmatrix} \begin{pmatrix}
		\bm 0_s \\
		\bm 1_{s}
		\end{pmatrix}^{\top}  = \begin{pmatrix}
		\bm I_s & \bm 0_{s, s}\\
		\bm 0_{s,s}& \bm 1_{s, s}
		\end{pmatrix},
		\end{align*}
		where for each $i\in [s]$, vector $\bm e_i \in \Re^n$ denotes the $i$th column of the identity matrix $\bm I_n$.
	\end{example}
	
	In Example~\ref{eg8}, the greedy Algorithm~\ref{algo:spca_greedy} outputs the solution $\hat{S} = [s]$ with the objective value of SPCA \eqref{eq_spcacom} equal to $k$. However, the true optimal value SPCA \eqref{eq_spcacom} is $||\bm A_{[s+1, n], [s+1, n]}||_{(k)} =  s$, which thus indicates that the $k/s$-approximation ratio is tight when $s\le n/2$.
	
	Now, let us focus on deriving the approximation ratio of the local search Algorithm~\ref{algo:spca_localsearch}. As an improved heuristic method, the local search Algorithm~\ref{algo:spca_localsearch} produces at least $k/s$-approximation ratio. Furthermore, in Example~\ref{eg8}, the local search Algorithm~\ref{algo:spca_localsearch} cannot improve the solution $\hat{S} = [s]$ from the greedy Algorithm~\ref{algo:spca_greedy} and thus the $k/s$-approximation ratio is also tight.
	
	Plugging $m=n$ and $s_1=s_2$ into the time complexities for greedy Algorithm~\ref{algo:svd_greedy} and local search Algorithm~\ref{algo:svd_localsearch} of SSVD, we obtain the corresponding time complexities for SPCA. \qed
\end{proof}


\section{Branch-and-Cut Algorithms for Solving SPCA and SSVD to Optimality}
\label{sec:bc}
In this section, we develop branch-and-cut algorithms with closed-form cuts to solve SPCA \eqref{eq_spca} and SSVD \eqref{eq_ssvdcom}, respectively.
\subsection{Branch-and Cut-Algorithm for SPCA}
The  branch-and-cut algorithm is built on an equivalent MILP formulation of SPCA as below. 
	\begin{restatable}{theorem}{thmbc}\label{thm_bc}
	The	SPCA \eqref{eq_spcasdp2} is equivalent to
	\begin{align}\label{eq_spcabc}
	\text{\rm (SPCA)} \ \   w^{spca}:=\max_{\bm z \in \{0,1\}^n, w \in \Re}\bigg\{w: \sum_{ i \in [n]}z_i \le s, w\le ||\bm A_{S, S}||_{(k)} + \sum_{i\in [n]\setminus S} A_{ii} z_i, \forall S\subseteq [n], |S|\le s \bigg\}.
	\end{align}
\end{restatable}
\begin{proof}
See Appendix \ref{proof:thm_bc}. \qed
\end{proof}

Note that SPCA \eqref{eq_spcabc} is amenable to the branch-and-cut method. That is, we start with a relaxed master problem with a small subset of constraints in SPCA \eqref{eq_spcabc} and given an incumbent solution $\hat{\bm z}$ with its support $\hat{S}$, we add the most violated constraint (i.e.,
$
w \le ||\bm A_{\hat{S}, \hat{S}}||_{(k)} +\sum_{i\in [n]\setminus \hat{S}} A_{ii}  {z}_i
$)
into the relaxed master problem and then repeat the solution procedure until the objective value of the relaxed problem is equal to the best lower bound. This solution procedure will terminate after a finite number of iterations.

\subsection{Branch-and-Cut Algorithm for SSVD}
Using the Cholesky decomposition of matrix $\bm A$, SPCA \eqref{eq_spca} possesses an MISDP \eqref{eq_spcasdp2}, which enables us to develop a branch-and-cut algorithm. However, this is difficult to directly extend to SSVD \eqref{eq_ssvdcom} since the matrix $\bm A$ may not even be symmetric, nor does its Cholesky decomposition exist. Fortunately, the following key lemma helps address this difficulty.
\begin{lemma}\label{lem:aug}
	For any matrix $\bm B \in \Re^{m\times n}$, let us define its augmented version $\overline{\bm B}\in \Re^{m+n, m+n}$ as below
	$$\overline{\bm B}\in \Re^{m+n, m+n} := \begin{pmatrix}
	\bm 0_{m, m} & \bm B\\
	\bm B^{\top} & \bm 0_{n, n}
	\end{pmatrix} + t \bm I_{m+n},$$ 
	for some constant $t\in \Re$. Then we have 
	\begin{enumerate}[(i)]
		\item The eigenvalues of matrix $\bm B$ are equal to
		\begin{align*}
		\lambda_i(\overline{\bm B}) = \begin{cases}
		t + \sigma_{i}(\bm B) , & \forall i\in [r],\\
		t, & \forall i \in [r+1, m+n-r],\\
		t - \sigma_{m+n+1-i}(\bm B), & \forall i \in [m+n-r+1, m+n];
		\end{cases}
		\end{align*}
		\item The matrix $\overline{\bm B}$ is positive semidefinite if $t\ge \sigma_{1}(\bm B)$.
	\end{enumerate}
\end{lemma}
\begin{proof}
	Part (i) directly follows from the result in Chapter 3.2. of \cite{ben2011lectures}. 
	
	Part (ii) follows from the fact that matrix $\overline{\bm B}$ is symmetric and the smallest eigenvalue is $t - \sigma_1(\bm B) \ge 0$. \qed
\end{proof}
By constructing the augmented counterpart of matrix $\bm A$, we can derive another equivalent combinatorial formulation of SSVD \eqref{eq_ssvdcom} as below.
\begin{restatable}{proposition}{prop} \label{prop3}
	The SSVD \eqref{eq_ssvdcom} is equivalent to 
	\begin{align} \label{eq_ssvdcom2}
	\text{\rm (SSVD)} \quad w^*:=	\max_{\begin{subarray}{c}
		S\subseteq [m+n]
		\end{subarray}
	} \left\{||\overline{\bm A}_{S, S}||_{(k)}:   |S\cap [m]|\le s_1, |S\cap [m+1, m+n]|\le s_2  \right\} - k \sigma_{1}(\bm A),
	\end{align}
	where matrix $\overline{\bm A}\in \S_+^{m+n}$ is defined by
	$$\overline{\bm A}\in \S_+^{m+n}:= \begin{pmatrix}
	\bm 0_{m, m} & \bm A\\
	\bm A^{\top} & \bm 0_{n, n}
	\end{pmatrix} + \sigma_{1}(\bm A) \bm I_{m+n}.$$
\end{restatable}
\begin{proof}
See Appendix \ref{proof:prop3}. \qed
\end{proof}

According to Part (ii) in Lemma~\ref{lem:aug},  matrix $\overline{\bm A}$ defined in Proposition~\ref{prop3} is positive semidefinite. Let  $\overline{\bm A} = \overline{\bm C}^{\top} \overline{\bm C}$ denote its Cholesky decomposition, where $\overline{\bm C} \in \Re^{\overline{r} \times (m+n)}$ is the corresponding Cholesky matrix, $\overline{r}$ is the rank of matrix $\overline{\bm A}$ and for each $i\in [m+n]$, let $\overline{\bm c}_i\in \Re^{\overline{r}}$ denote its $i$th column.  With the notation introduced, a second MISDP formulation of SSVD \eqref{eq_ssvdcom2} can be derived as below.
\begin{corollary} \label{corsvd}
	The SSVD \eqref{eq_ssvdcom2} is equivalent to
	\begin{align}\label{eq_svdsdp2}
	\text{\rm (SSVD)} \quad w^*:=	\max_{ \begin{subarray}{c} \bm z \in \{0,1\}^{m+n}, \bm X\in \S_+^{\overline{r}},\\
		\bm W_1, \cdots, \bm W_{m+n} \in \S_+^{\overline{r}}
		\end{subarray}
	} \bigg\{ &\sum_{i\in [m+n]} \overline{\bm c}_i^{\top} \bm W_i \overline{\bm c}_i: \bm X \preceq \bm I_{\overline{r}}, \tr(\bm X) \le k, \bm X \succeq \bm W_i,  \tr(\bm W_i) \le z_i,  \notag\\
	& \forall i \in [m+n],  \sum_{i\in [m]} z_i \le s_1,   \sum_{i\in [m+1, m+n]} z_i \le s_2  \bigg\} - k \sigma_1(\bm A).
	\end{align}
\end{corollary}
\begin{proof}
	The proof is almost identical to that of Corollary~\ref{cor2} and is thus omitted. \qed
\end{proof}

Based on the MISDP \eqref{eq_svdsdp2}, we can also derive an MILP formulation of SSVD, which is the key to designing a branch-and-cut algorithm.
\begin{corollary} \label{corsvdbc}
	The SSVD \eqref{eq_svdsdp2} is equivalent to
	\begin{align}\label{eq_svdbc}
	\text{\rm (SSVD)} \quad w^*:=\max_{\bm z \in \{0,1\}^{m+n}, w}\bigg\{w: \sum_{ i \in [m]}z_i \le s_1, \sum_{ i \in [m+1, m+n]}z_i \le s_2, 
	w\le ||\overline{\bm A}_{S, S} ||_{(k)} +  \notag \\
	\sum_{i\in [m+n]\setminus S} \overline{A}_{ii} z_i,
	\forall S\subseteq [m+n], |[m]\cap S|\le s_1, |[m+1, m+n]\cap S|\le s_1  \bigg\} -k \sigma_1(\bm A).
	\end{align}
\end{corollary}
\begin{proof}
	The proof is almost identical to that of Theorem~\ref{thm_bc} and is thus omitted. \qed
\end{proof}
Similarly, we can apply the branch-and-cut method to solve SSVD \eqref{eq_svdbc}. That is, we start with a  master problem with a small subset of constraints in SSVD \eqref{eq_svdbc} and given an incumbent solution $\hat{\bm z}$ with support $\hat{S}$, we add the most violated constraint (i.e.,
$
w\le ||\overline{\bm A}_{\hat{S}, \hat{S}} ||_{(k)} + 
\sum_{i\in [m+n]\setminus \hat{S}} \overline{A}_{ii}  z_i
$),
into the relaxed master problem and then repeat the solution procedure until the objective value of the relaxed problem is equal to the best lower bound. 

\section{Main Numerical Results}\label{sec:data}
In this section, we implement our proposed exact and approximation algorithms to solve synthetic and real-world instances of SSVD \eqref{eq_ssvdcom} and SPCA \eqref{eq_spca}. Specifically, we consider four real-world datasets from in \cite{dey2018convex}, a drug abuse dataset from the National Survey on Drug Use and Health in 2018, and a solar flare dataset  in \cite{yan2018real}, which are described in Table~\ref{table:data}. 

\begin{table}[ht]
	\setlength{\tabcolsep}{3pt} 
	\renewcommand{\arraystretch}{1} 
	\caption{Statistics of Real Datasets }
	\label{table:data}
	\begin{center}
			\vskip -0.15in
		\begin{sc}
			\begin{tabular}{l |c c c  c r}
				\hline
				Dataset & $m$ & {$n$} & Rank  & Positive semidefinite  \\
				\hline
				Eisen-1 &  79& 79 & 79 & Yes\\
				Eisen-2 &  118& 118 & 78 & Yes \\
				Colon &  500& 500 & 62  & Yes\\
				Reddit & 2000 & 2000 & 949 & Yes\\
				Drug & 16313 & 2365 & 2224 & No\\
				MOUSSE & 232 &292 & 232 & No\\
				\hline
			\end{tabular}
		\end{sc}
	\end{center}
	\vskip -0.25in
\end{table}
All the experiments are conducted in Python 3.6 with calls to Gurobi 9.0 on a personal PC with 2.3 GHz Intel Core i5 processor and 8G of memory. 
Note that  the Ky Fan $k$-norm of matrices is efficiently computed by the PRIMME package  \citep{wu2014primme}.

\subsection{Illustration of Sparsity Levels on Synthetic and Real Datasets} 
This subsection illustrates how our SSVD model \eqref{eq_ssvdcom} performs against the sparsity levels (i.e., the tuning parameters $(s_1, s_2)$) using synthetic and real-world data. Particularly, the synthetic data matrices $\bm A \in \Re^{m\times n}$  are obtained by sampling $m\times n$ i.i.d. random standard normal variables.

To evaluate SSVD \eqref{eq_ssvdcom} and illustrate the selection of sparsity levels, we compute the ratio of its optimal value over the Ky Fan $k$-norm of the original data matrix $\bm A$, i.e., $w^*/||\bm A||_{(k)}$ for various parameters $s_1,s_2$. Here, the larger the ratio means that the more information that the optimal submatrix  from SSVD contains. Using the MILP formulation \eqref{corsvdbc} of SSVD, we develop a branch-and-cut algorithm to exactly compute the optimal values for the synthetic dataset with $(m,n,k)=(8,10,3)$. Figure~\ref{ssvd1a_p} presents how the ratio $w^*/||\bm A||_{(k)}$ varies as $s_1$ and $s_2$ change.
Note that the running time of all the cases is less than one minute.

\begin{figure}[htbp]
	\subfloat[The ratio $w^*/||\bm A||_{(3)}$ with $(m,n,k)=(8, 10, 3)$]{
		{\includegraphics[width=0.45\columnwidth]{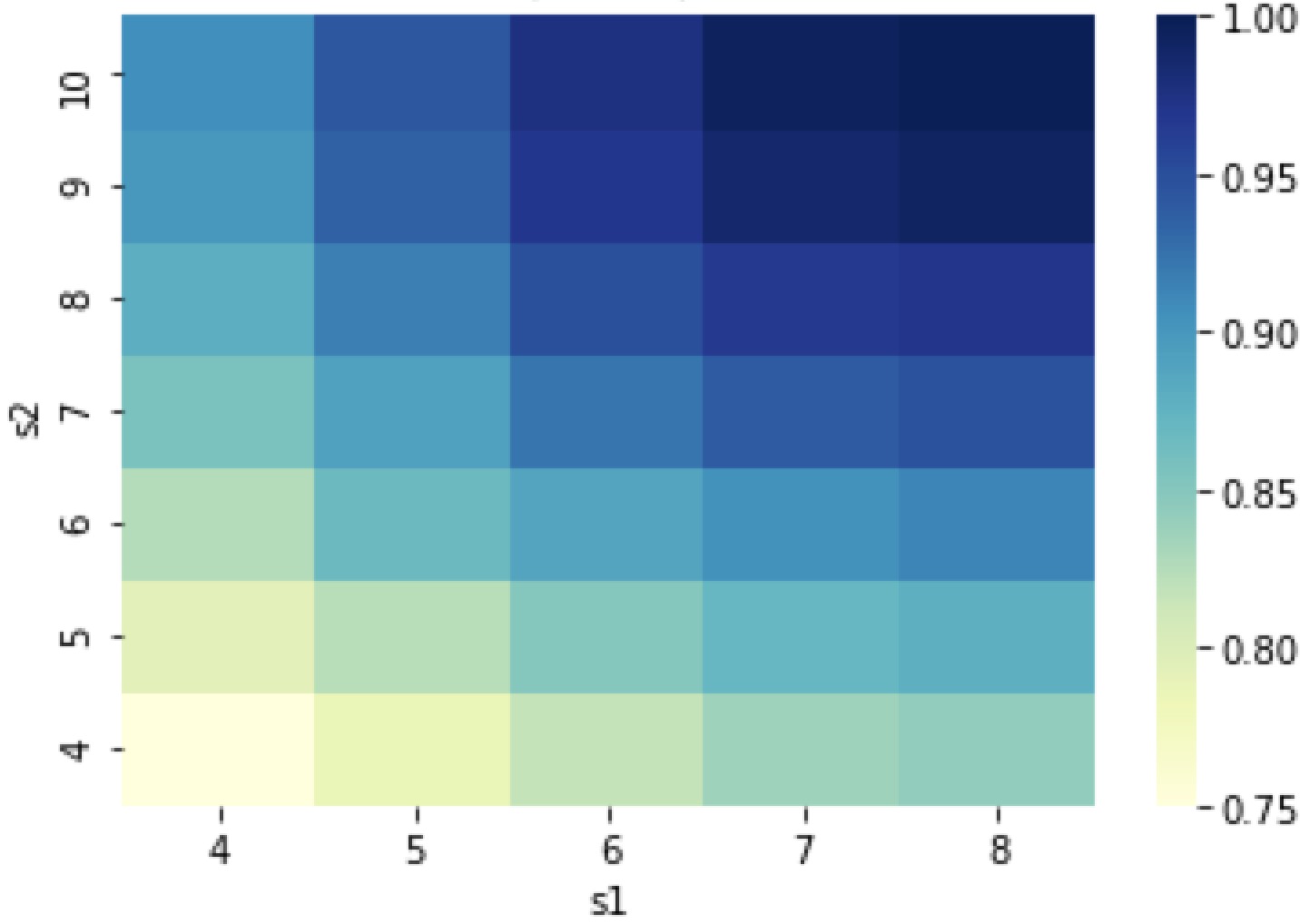}}	\label{ssvd1a_p}}
	~
	\subfloat[The ratio $\hat{w}_1/||\bm A||_{(5)}$ with $(m,n,k)=(200, 250, 5)$]{
		{\includegraphics[width=0.45\columnwidth]{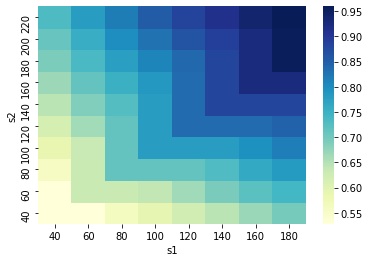}}			\label{ssvd1b_p}}
	
	\subfloat[The ratio $\hat{w}_1/||\bm A||_{(5)}$ with $(m,n,k)=(118, 118, 5)$]{
		{\includegraphics[width=0.45\columnwidth]{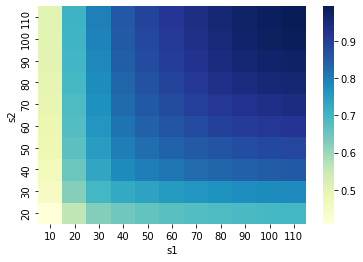}}	\label{ssvd1ra_p}}
	~
	\subfloat[The ratio $\hat{w}_1/||\bm A||_{(7)}$ with $(m,n,k)=(500, 500, 7)$]{
		{\includegraphics[width=0.45\columnwidth]{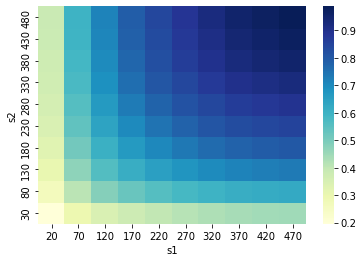}}			\label{ssvd1rb_p}}
	\caption{Illustration of Sparsity Levels}
	\label{figure}
		\vskip -0.15in
\end{figure}

We also test the larger instances, where we employ the scalable approximation Algorithm~\ref{algo:trun2}, Algorithm~\ref{algo:trun3}, and Algorithm~\ref{algo:svd_greedy}, and return the best objective value as the lower bound of SSVD \eqref{eq_ssvdcom}, denoted by $\hat{w}_1$. 
Accordingly, we compute the approximate ratio by $\hat{w}_1/ ||\bm A||_{(k)}$. Figure~\ref{ssvd1b_p} illustrates the ratio against parameters $s_1$ and $s_2$ for the synthetic instance of $m=200$, $n=250$, and $k=5$.  In addition, we test Eisen-2 and Colon datasets in	Table~\ref{table:data} with $k=5$ and $k=7$, respectively, as displayed in Figure~\ref{ssvd1ra_p} and Figure~\ref{ssvd1rb_p}. 
Note that the running time of approximation algorithms on all the cases is also less than one minute. 
For the two {synthetic and real} datasets in Figure~\ref{figure}, we can {observe}
that the optimal submatrix from SSVD \eqref{eq_ssvdcom} captures about $80\%$ of the information in matrix $\bm A$ when $s_1\approx m/2$ and $s_2\approx n/2$, i.e., a quarter size of the original matrix $\bm A$.  
%

\subsection{Evaluations of the Algorithms on Real Datasets}
In this subsection, we evaluate our proposed exact and approximation algorithms on 

We apply the proposed exact and approximation algorithms to solving SPCA \eqref{eq_spca} on the first three positive semidefinite matrices with various parameters $s$ and $k$. 
Table~\ref{table1} presents  the numerical results, where ``Paras" denotes the parameters,  ``B\&C'' denotes the branch-and-cut algorithm,   "Alg" denotes the algorithm, the optimality gap(\%) is computed by
$$ \text{gap}(\%) = 100\times \frac{\text{optimal value} -\text{ lower bound}}{\text{lower bound}},$$
``LS" denotes the benchmark, a modified local search algorithm proposed by \cite{dey2020solving}, which uses a random initial solution, 
and ``Ran'' and ``Vol" denote the  randomized leverage-score algorithm in \cite{mahoney2009cur} and  maximum volume-based algorithm in \cite{mikhalev2018rectangular} for submatrix selection, respectively.
Note that the computational time of the approximation Algorithm~\ref{algo:trun2spca}-Algorithm~\ref{algo:spca_localsearch} are at most one second and thus omitted.
It can be seen that the branch-and-cut method is able to solve small- or medium-sized instances to optimality, which allows us to test the solution quality of the approximation algorithms by computing their optimality gaps as shown in Table~\ref{table1}. It is worth mentioning that our selection Algorithm~\ref{algo:trun2spca} and local search Algorithm~\ref{algo:spca_localsearch} are consistently better than the benchmark. 
We also see that greedy Algorithm~\ref{algo:spca_greedy} and  local search Algorithm~\ref{algo:spca_localsearch} returns the optimal solution with zero optimality gap for most of the instances.  Lastly, it is interesting to note that the solution quality of the truncation Algorithm~\ref{algo:trun1spca} is more sensitive to the parameter $s$ than $k$, which complies with its theoretical approximation ratio.

\begin{table}[htb]
	\setlength{\tabcolsep}{1.7pt} 
	\renewcommand{\arraystretch}{1} 
	\caption{Exact and Approximation Algorithms for Solving SPCA \eqref{eq_spca}}
	\label{table1}
	\begin{center}
				\vskip -0.15in
		\begin{sc}
			\begin{tabular}{r |r r |r  r |r r r r r r r r r   }
				\hline
				\multirow{2}{*}{Data} & \multicolumn{2}{c|}{Paras} & \multicolumn{2}{c|}{B\&C} & \multicolumn{9}{c}{Gap(\%) of Approximation Algorithms}  \\
				& $s$ & {$k$} &  \multicolumn{1}{c}{$w^{spca}$} & time(s) & Alg \ref{algo:trun1spca} & time(s) & \multicolumn{1}{c}{Alg \ref{algo:trun2spca}} & \multicolumn{1}{c}{Alg \ref{algo:trun3spca}} & \multicolumn{1}{c}{Alg \ref{algo:spca_greedy}} & \multicolumn{1}{c}{Alg \ref{algo:spca_localsearch}} & \multicolumn{1}{c}{LS} & \multicolumn{1}{c}{Ran} & \multicolumn{1}{c}{Vol} \\
				\hline
				Eisen-1 &   5& 2 & 15.9771 & 1 & 0.00 & 1 & 0.00 & 0.00 & 0.00 & 0.00 &0.00 & 35	& 40\\
				$n=79$ & 5& 3 & 16.1501 & 1 &  0.00 & 1 & 0.00 & 0.00 & 0.00 & 0.00 &  0.00 & 18	& 40 \\
				& 5& 4 & 16.2399&  1 &  0.00 & 1 & 0.00 & 0.00 & 0.00 & 0.00 &  0.00 & 15	& 39 \\
				& 10& 2 & 20.1578 &  1 &  0.00 & 1 & 0.00 & 2.50 & 0.00 & 0.00 &  4.58 & 53	& 16 \\
				& 10& 3 & 20.5690 & 1 &   0.00 & 1 & 0.00 & 2.30 & 0.00 & 0.00 &  4.03 & 44	& 14\\
				& 10& 4 & 20.8163 & 1 &   0.00 & 1 & 0.00 & 2.30 & 0.00 & 0.00 &  3.85 & 33	&14\\
				& 15& 2 & 20.7619 & 180 &   0.96 & 1 & 0.33 & 3.65 & 0.00 & 0.00 &  6.80  & 29	&2\\
				& 15& 3 & 22.2571 & 1 & 3.25 & 1 & 0.00 & 5.92 & 0.00 & 0.00 &  10.84 & 25	&5\\
				& 15 & 4 & 22.7143 & 1 & 2.55 & 1 & 0.07 & 5.14 & 0.00 & 0.00 &  10.19 & 33	&5\\
				\hline
				Eisen-2 &   5& 2 & 8.1438 & 6 & 0.00 & 1 & 0.00 &19.75 &14.45 &0.93  &19.75 & 30	&35 \\
				$n=118$ & 5& 3 & 8.4338 & 10 & 2.08 & 1 &2.06 &22.70 &0.00 &0.00 &22.70  & 28&	32 \\
				& 5& 4 & 8.5467 & 10 & 2.87 & 1&2.87 &23.57 &0.00 &0.00 &23.57 & 8	&31 \\
				& 10& 2 & 13.8929 & 510 & 15.31 & 1 &0.83 &15.31 &7.71 &1.45 &6.30 & 7	&29\\
				& 10& 3 & 14.4792 &2095 & 18.21 & 1 & 0.15 & 17.83 & 0.00 & 0.00 &  8.51 & 65	&18\\
				& 10& 4 & 14.7313 & 3325 & 18.56 & 1 & 0.00 & 17.99 & 0.00 & 0.00 &  9.12 & 51	&17\\
				\hline
				Colon &   5& 2 & 1709.9575 & 10 &0.00 & 302 &0.00 &0.00 &0.00 &0.00 &7.55  & 169 & 83\\
				$n=500$ 	 & 5& 3 & 1718.1590 &25 &0.00 & 297&0.00 &0.00 &0.00 &0.00 &5.85 & 56 & 77\\
				& 5& 4 & 1720.1219 &20 &0.00 & 303&0.00 &0.00 &0.00 &0.00 &5.60 & 115 & 72 \\
				& 10& 2 & 2794.1404 &3951 & 0.38 &315 &0.00 &0.38 &0.00 &0.00 &0.21 & 98 & 114 \\
				\hline
			\end{tabular}
		\end{sc}
	\end{center}
				\vskip -0.25in
\end{table}


%

We further employ the proposed approximation algorithms to solve SSVD \eqref{ssvd} on all data matrices by selecting different parameters $s_1$ and $s_2$. 
The computational results can be found in Table~\ref{table2}, where ``LB" denotes the lower bound returned by an approximation algorithm. Note that for these SSVD instances, the algorithm developed by \cite{dey2020solving} cannot solve them since the selected submatrices might not be symmetric. In Table~\ref{table1}, we see that the other two existing algorithms perform poorly even when solving the small instances. We also notice that the branch-and-cut algorithm has difficulty in finding a feasible solution within one hour
Thus, neither the benchmark algorithm nor the branch-and-cut algorithm is reported.
It is worthy of mentioning that to more efficiently solve high-dimensional Drug dataset, we restrict  the selection Algorithm~\ref{algo:trun1} and local search Algorithm~\ref{algo:svd_localsearch} to a proper but larger-sized submatrix than $s_1\times s_2$ by first greedily selecting a larger-sized (e.g., $2s_1\times 2s_2$) submatrix and then applying these two algorithms to the submatrix.
To better illustrate the comparison of the approximation algorithms for SSVD \eqref{ssvd}, we normalize the lower bounds by the ones found by the local search Algorithm~\ref{algo:svd_localsearch} and display the results in Figure~\ref{fig:ssvd1}.
A larger ratio means a larger lower bound and a better solution.
We see that the local search Algorithm~\ref{algo:svd_localsearch} significantly improves the solution of greedy Algorithm~\ref{algo:svd_greedy} and consistently performs the best among all algorithms. 

\begin{table}[htb]
			\vskip -0.15in	
	\fontsize{10}{10}\selectfont 
	\caption{Approximation Algorithms for Solving SSVD \eqref{eq_ssvdcom}}
		\vskip -0.15in
	\label{table2}
	\setlength{\tabcolsep}{1pt} 
	\renewcommand{\arraystretch}{1} 
	\begin{center}
		\begin{sc}
			\setlength{\tabcolsep}{1.2pt}\renewcommand{\arraystretch}{1.2}
			
			\begin{tabular}{c |r  r r |  r  r | r  r | r r | r r |r r}
				\hline
				\multirow{2}{*}{Data}& \multicolumn{3}{c|}{Paras} & \multicolumn{2}{c|}{Alg \ref{algo:trun1}} & \multicolumn{2}{c|}{Alg \ref{algo:trun2}} & \multicolumn{2}{c|}{Alg \ref{algo:trun3}} & \multicolumn{2}{c|}{Alg \ref{algo:svd_greedy}} & \multicolumn{2}{c}{Alg \ref{algo:svd_localsearch}}\\
				&\multicolumn{1}{c}{ $s_1$}    &\multicolumn{1}{c}{ $s_2$} & \multicolumn{1}{c|}{ $k$} &\multicolumn{1}{c}{ LB} & time(s) & \multicolumn{1}{c}{ LB} & time(s) & \multicolumn{1}{c}{ LB}  & time(s) &\multicolumn{1}{c}{ LB}  & time(s) & \multicolumn{1}{c}{ LB}  & time(s)  \\
				\hline
				Eisen-1&   5&  10 & 3 &  17.89& 1& 17.89& 1 & 17.89& 1 & 16.15& 1 & 18.08& 1 \\
				$m=79$ & 10&  15 & 3 & 20.82 & 1& 20.97 & 1 & 20.32& 1 &20.57 & 1 & 21.11& 2  \\
				$n=79$  &15 & 20 & 3&  22.15& 1& 21.78 & 1 & 21.12 & 1 & 22.26 &  1 & 22.53 & 1\\
				\hline
				Eisen-2 &   15&  20 & 4 &  21.38& 1& 21.38 & 1 & 19.60 & 1 & 19.39& 1& 21.41 & 4  \\
				$m=118$ & 20&  25 & 4 & 25.24 & 1& 25.24 & 1 & 25.24 & 1 & 23.89 & 2 & 25.32 & 3  \\
				$n=118$  & 25 & 30 & 4&  28.02& 1& 28.09 & 1& 28.03 & 1 & 26.96&  2 & 28.83 & 8\\
				\hline	
				Colon &   30&  25 & 5 &  5943.40& 226& 5957.77 & 1 & 5943.40& 1 & 5621.61 & 13 & 5987.91 & 57  \\
				$m=500$ & 35&  30 & 5 & 6712.53 & 209& 6714.23 & 1 & 6712.53 & 1 & 6391.81 & 16 & 6724.51 & 23  \\
				$n=500$  & 40 & 35 & 5&  7387.09& 280& 7387.08 & 1 & 7387.08 & 1 & 7087.79 &  19 & 7405.59 & 27\\
				\hline
				Reddit &   40&  35 & 6 &  4260.76& 371& 4230.60 &6 & 3761.87 & 1 & 4273.96 & 92 & 4364.84 & 403  \\
				$m=2000$ &45&  40& 6 & 4634.98 & 340& 4409.92 & 6 & 3882.21 & 1 & 4461.45 & 104 & 4654.04& 1052  \\
				$n=2000$  & 50 & 45 & 6&  4786.25& 363& 4601.39 & 8 & 3963.54 & 1 & 4734.31&  132 & 4809.57 & 922\\
				\hline
				Drug &   30& 35 & 7 & 827.21 & 554 & 821.38& 203 &766.41& 1 &982.74&163 &1012.91  &549\\
				$m=16313$  & 35& 30 & 7 &  873.98 & 722 & 862.99 & 173 &808.94 & 1 &982.74&236 &1067.52&724\\
				$n=2365$ & 50& 45 & 7 & 1070.86 & 1442 & 1064.35& 253 &997.17& 1 &1247.86&412 &1327.95  &1332\\
				& 45& 50 & 7 &  1034.86 & 1205 & 1036.71 & 271 &961.55& 1 &1247.86&331 &1286.00 &1121\\
				\hline
			\end{tabular}
		\end{sc}
	\end{center}
	\vskip -0.2in
\end{table}

\begin{figure}[ht]
	\vskip -0.1in
	\centering
	\includegraphics[width=0.9\columnwidth]{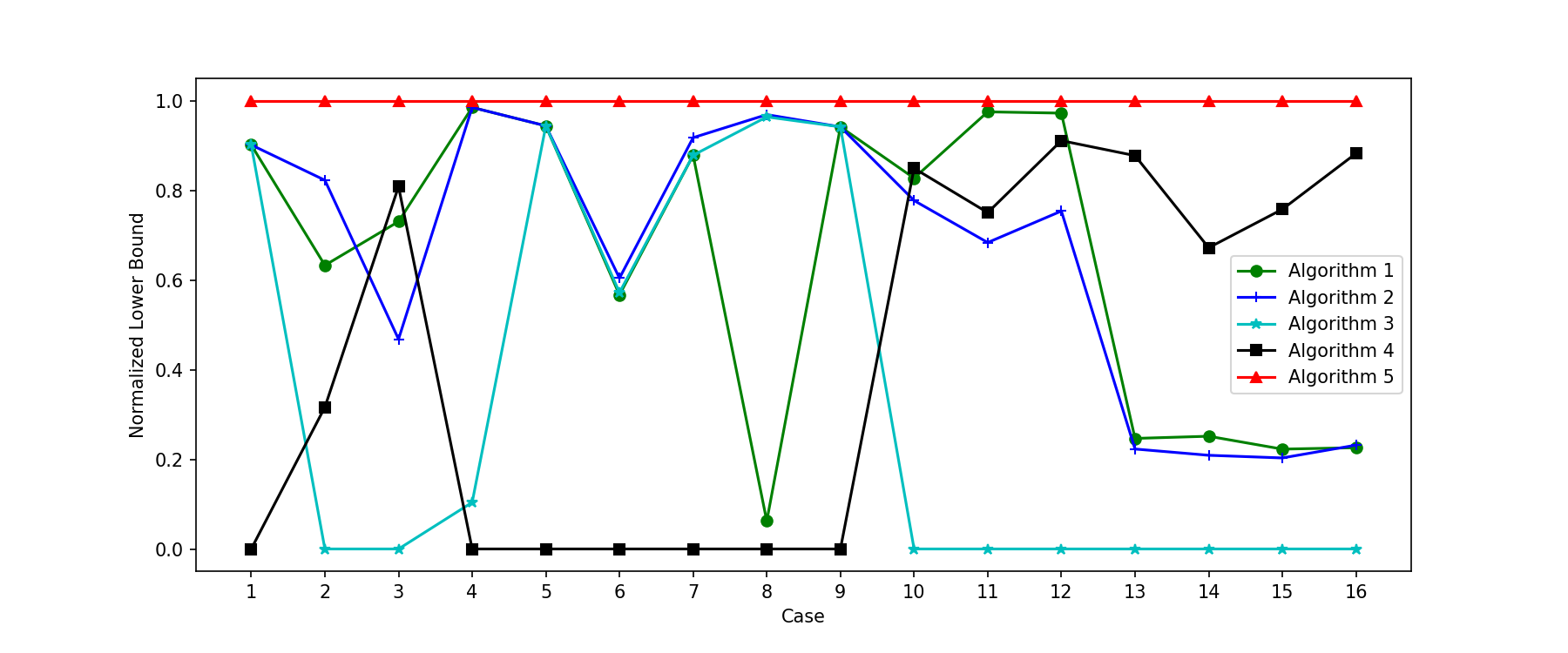}		
			\vskip -0.1in
	\caption{Normalized Lower Bounds in Table~\ref{table2}}
	\label{fig:ssvd1}
	\vskip -0.1in
\end{figure}

\subsection{Anomaly Detection by SSVD: Solar Flare Detection}

Recently, the low-rank-based methods \citep{qu2018hyperspectral,  cheng2019graph} or low-rank-embedded deep learning methods \citep{jiang2021lren,song2019hyperspectral} have been widely  applied to anomaly detection, especially in hyperspectral imagery \citep{zhang2021spectral}, where the basic idea is that the normal data can be recast in a low-dimensional subspace and  admit low-rank property after removing the anomalies. Since our SSVD \eqref{ssvd} achieves simultaneously sparse and low-rank approximation, it can be directly applied to the solar flare detection. 
Particularly, the solar flare phenomenon, a  rapid and intense variation of brightness in the solar atmosphere lasting from minutes to hours, may potentially pose risks to electronic  power grids and radio communications and thus needs efficiently detecting. 

The  solar dataset-MOUSSE in Table~\ref{table:data} used for the anomaly detection of SSVD \eqref{ssvd} consists of  300 solar images in a video format collected from a satellite, where  each image is of size
$232\times 292$ pixels \citep{yan2018real}. Therefore,  there are 300 temporal  data matrices of equal size, $232\times 292$. Based on the previous numerical performance, we choose to adopt the truncation Algorithm~\ref{algo:trun2} of SSVD to find a sparse and low-rank approximation for each matrix by setting $s_1=100$, $s_2=100$, and $k=10$, where the output values of  truncation Algorithm~\ref{algo:trun2} representing the  approximation results of matrices at all time points are displayed in Figure~\ref{ssvd_solar}. For comparison purposes, the sum of all the singular values of each data matrix  is shown in Figure~\ref{solar} without any  sparse or low-rank approximation. Both methods can be implemented within several minutes. In the temporal solar dataset, two  occurrences of solar flares at intervals $187\sim 202$ and $216\sim268$ have been recognized by  \cite{nabhan2021correlation}. It is seen  that the sparse and low-rank approximation by SSVD \eqref{ssvd} successfully identifies the potential solar flares at red points $177\sim 201$ and $217\sim267$ in Figure~\ref{ssvd_solar}, which is slightly different but consistent with the literature. On the other hand, we can only identify the solar flare occurrence at $217\sim236$ if we simply use the original data matrix for detection in Figure~\ref{solar}.

\begin{figure}[htbp]
		\vskip -0.15in
	\subfloat[{Objective value of SSVD with $(s_1,s_2,k)=(100, 100, 10)$ at each time point}]{
		{\includegraphics[width=0.45\columnwidth]{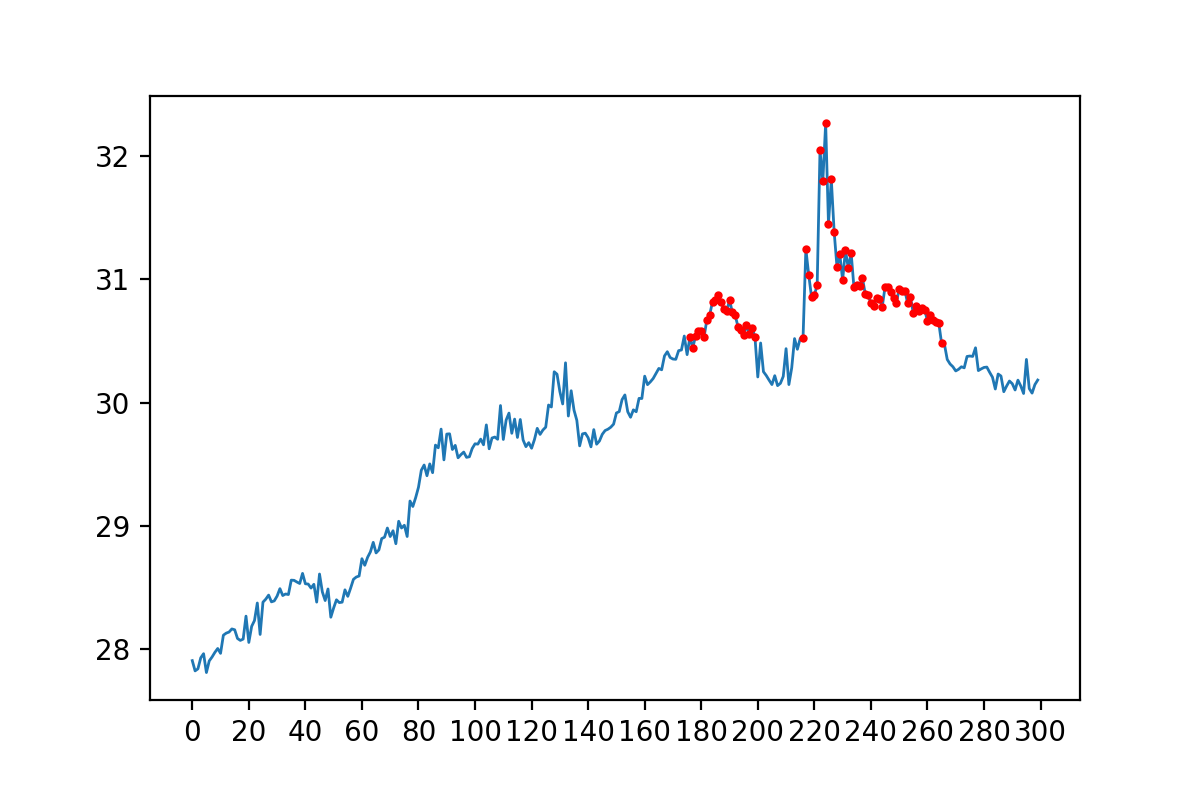}}	\label{ssvd_solar}}
	~
	\subfloat[Nuclear norm of the original  data matrix at each time point]{
		{\includegraphics[width=0.45\columnwidth]{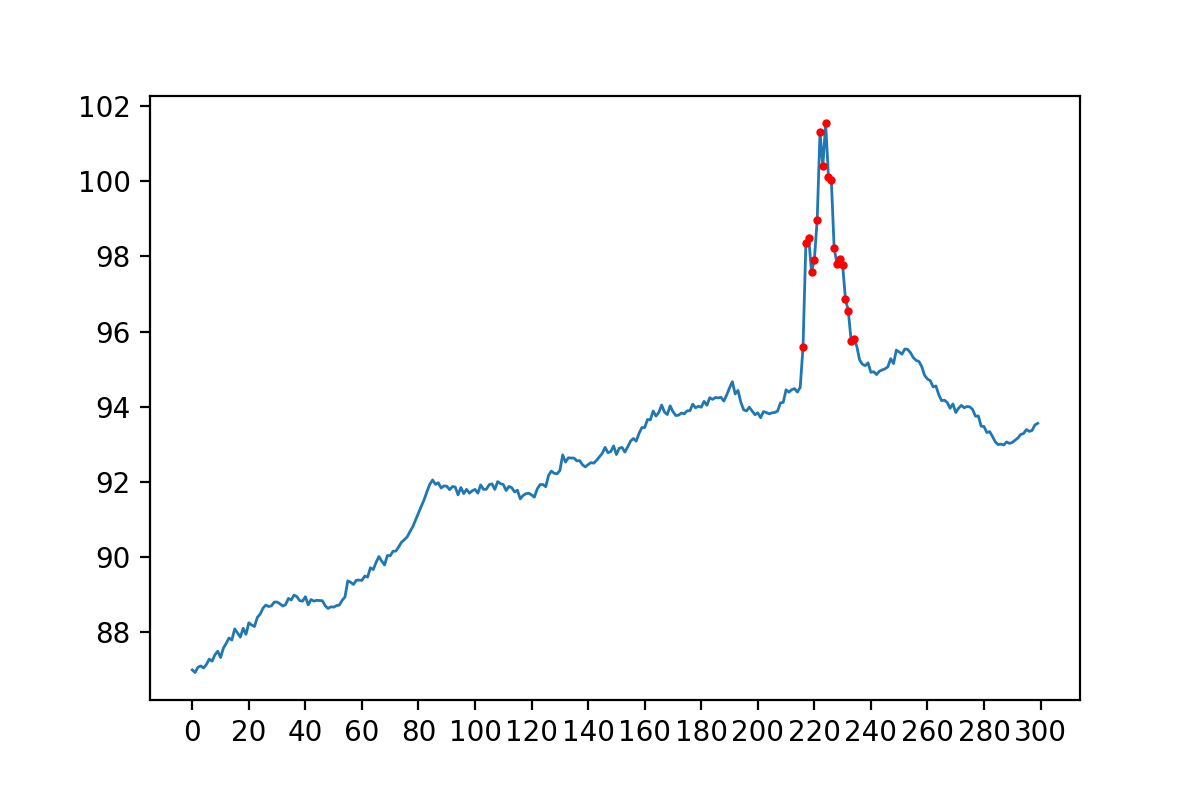}}			\label{solar}}
		\vskip -0.1in
	\caption{Sparse and Low-rank Approximation of SSVD for Solar Flare Detection}
	\label{fig:solar}
		\vskip -0.15in
\end{figure}

As seen in Figure~\ref{fig:solar}, the proposed SSVD model  is not only more effective but also improves interpretability and offers insights of the anomaly detection results. Figure~\ref{fig:solarimage} shows comparisons of original images and the selected sub-image by the SSVD at two time points, 192 and 222, respectively.
We see that at  both time points, the selected sub-images capture the abnormal brightness and remove a large portion of the irrelevant dark solar atmosphere. These selected sub-images can be useful for domain experts to pinpoint the important sub-image and study the flares more closely. This demonstrates that the SSVD can succeed in extracting valuable information and selecting crucial features. 

\begin{figure}[ht]
		\vskip -0.2in
	\subfloat[Original  image  at 192 ]{
		{\includegraphics[width=0.45\columnwidth]{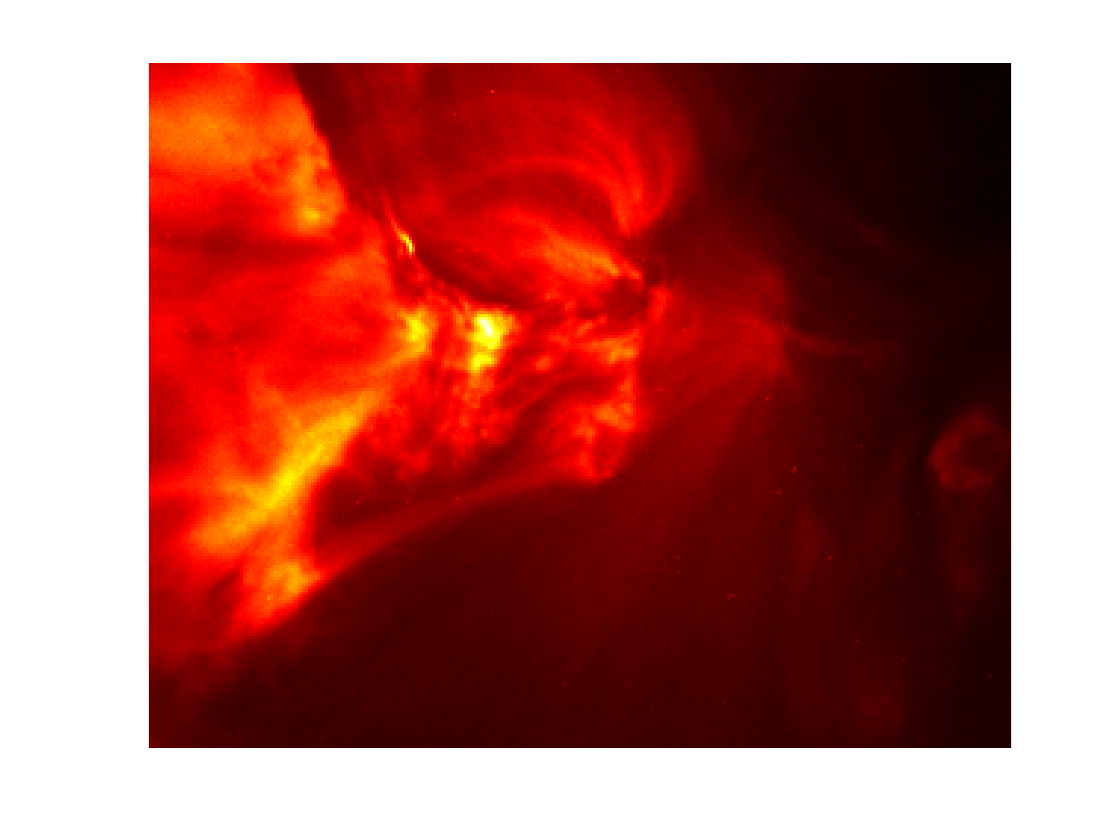}} 			\label{flare11}}
	~
	\subfloat[ Sub-image extracted by SSVD at 192]{
		{\includegraphics[width=0.45\columnwidth]{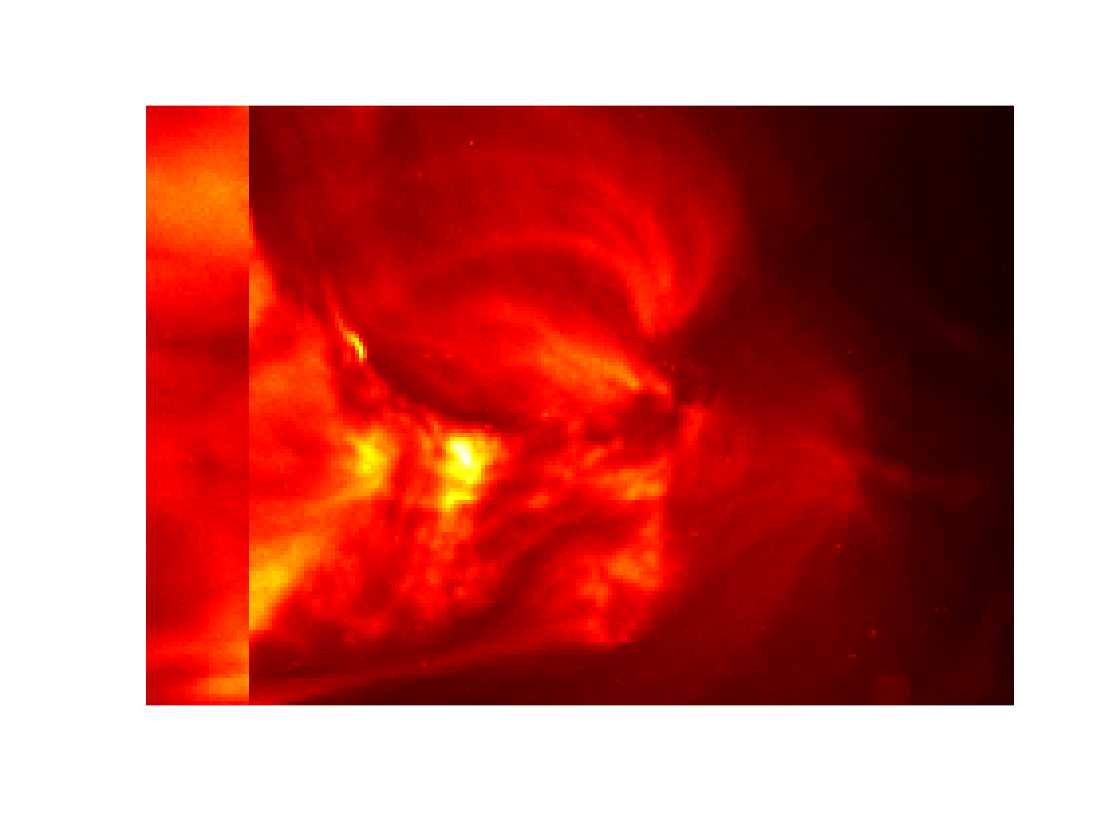}}			\label{flare12}}
	
	\subfloat[Original  image  at 222]{
		{\includegraphics[width=0.45\columnwidth]{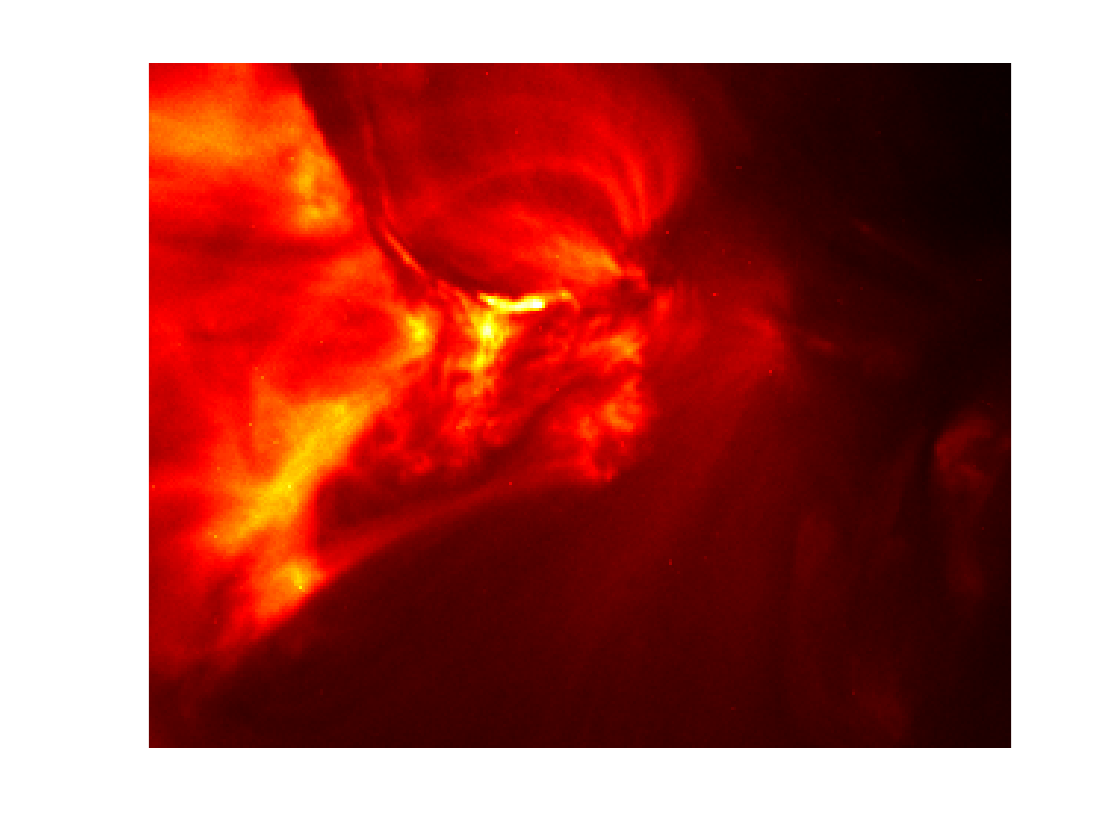}}	\label{flare21}}
	~
	\subfloat[Sub-image extracted by SSVD at 222]{
		{\includegraphics[width=0.45\columnwidth]{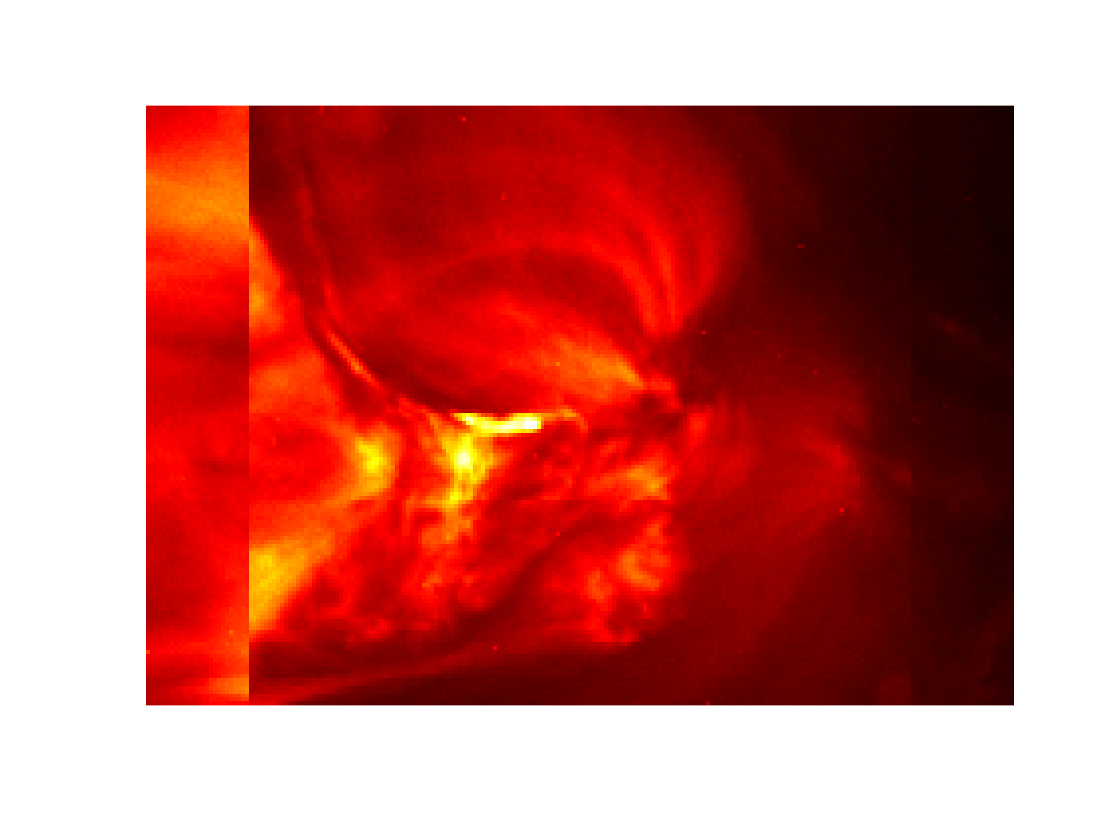}}			\label{flare22}}

	\caption{Approximation of Solar Flare Images  by SSVD}
	\label{fig:solarimage}
		\vskip -0.3in
\end{figure}

	\section{Conclusions}
\label{sec:conclusions}
For the best submatrix selection within a given low rank, we present a novel SSVD model, which ensures the selected submatrix to be simultaneously sparse and low-rank. Future interesting directions on {SSVD} include {studying how to finalize the effective sparsity levels and} developing efficient convex relaxation algorithms, robust SSVD under noisy data, fair SSVD among multiple sensitive groups, sparse and low-rank tensor decomposition, and so on.
\begin{acknowledgements}
	This research has been supported in part by the National Science Foundation grants 2046426 and 2153607. The authors would like to thank Dr. Hao Yan from Arizona State University for sharing the solar flare dataset.
\end{acknowledgements}

\bibliography{references.bib}

\newpage
\titleformat{\section}{\large\bfseries}{\appendixname~\thesection .}{0.5em}{}
 
\begin{appendices}
	\section{Proofs}\label{proofs}



\subsection{Proof of Proposition~\ref{props1}}	\label{proof:props1}
\propsone*
\begin{proof}
	\begin{subequations}
		We split the proof into two steps: NP-hardness and MILP reformulation.
		\begin{enumerate}[(i)]
			\item 	\textbf{NP-hardness.}
			
			We show that problem \eqref{eq_alter} can be reduced to the well-known NP-hard maximum clique problem.
			
			Given an undirected graph $G=(V, E)$,  let $n=|V|$ denote the number of vertices. Then we construct matrix $\bm A \in \Re^{n\times n}$ as below
			\begin{align*}
			A_{ij} = \begin{cases}
			1, & \text{if } (i,j)\in E\\
			n^2, & \text{if } i=j\\
			0, & \text{otherwise}
			\end{cases}, \forall i\in [n],j\in [n].
			\end{align*}
			
			We see that matrix $\bm A$ is symmetric and diagonally dominant and is thus positive semidefinite.
			
			Suppose $s_1=s_2 = s \in [n]$. For any solution $(S_1^*, S_2^*)$ to problem \eqref{eq_alter},  if $S_1^*=S_2^*$, then $s$ diagonal elements are chosen and the objective value is at least $sn^2$. On the other hand, if $S_1^*\neq S_2^*$, at most $s-1$ diagonal elements are chosen and the objective value is at most $s(s-1)+(s-1)n^2 < sn^2$. Therefore, letting $(S_1^*, S_2^*)$ denote the optimal solution to  problem \eqref{eq_alter}, at optimality, we must have $S_1^* = S_2^*=S^*$. That is,  problem \eqref{eq_alter} reduces to
			\begin{align} \label{eq_alter_cl}
			\ \max_{
				\begin{subarray}{c}
				S\subseteq[n]
				\end{subarray}
			} \left\{||\bm A_{S, S}||_F: |S| \le s\right\}.
			\end{align}
			
			Next, we will show that finding the largest $s$ such that the optimal value of problem \eqref{eq_alter_cl} is equal to $s(s-1)+sn^2$ is equivalent to finding the maximum clique in the graph $G=(V, E)$. Indeed, for each $s\in [n]$, if the optimal value of problem \eqref{eq_alter_cl} is equal to $s(s-1)+sn^2$, then we find a size-$s$ clique with vertices indexed by $S^*$. The largest such $s$ provides us a lower bound of the maximum clique of the graph $G=(V, E)$. On the other hand, given a size-$s$ clique with its vertices indexed by $S^*$, clearly $S^*$ is an optimal solution to problem \eqref{eq_alter_cl} with the optimal value $s(s-1)+sn^2$. Hence, the size of maximum clique in the graph $G=(V, E)$ is smaller than or equal to the largest $s$ such that the optimal value of problem \eqref{eq_alter_cl} is equal to $s(s-1)+sn^2$.
			
			\item \textbf{MILP reformulation.}
			
			First, for any feasible subsets $S_1$ and $S_2$ to problem \eqref{eq_alter}, the objective value is equal to
			\begin{align*}
			||\bm A_{S_1, S_2}||_F = \sqrt{\sum_{i\in S_1}\sum_{j\in S_2} A_{ij}^2},
			\end{align*}
			where the equality is because of the definition of Frobenius norm. Thus,  by taking square of the objective function, problem \eqref{eq_alter} can be equivalently reformulated as
			\begin{align} \label{eq_alter1}
			\ \max_{
				\begin{subarray}{c}
				S_1\subseteq[m],  S_2\subseteq [n]
				\end{subarray}
			} \bigg\{\sum_{i\in S_1}\sum_{j\in S_2} A_{ij}^2: |S_1| \le s_1,|S_2| \le s_2\bigg\}.
			\end{align}
			
			Next, let us introduce binary variables $\bm x \in \Re^m$ and $\bm y \in \Re^n$ such that $x_i$=1 if $i\in S_1$ and zero, otherwise for all $i\in [m]$ and   $y_j$=1 if $j\in S_2$ and zero, otherwise for all $j\in [n]$.  Problem \eqref{eq_alter1} can be recast as the following binary bilinear program
			\begin{align*}
			\max_{
				\begin{subarray}{c}
				\bm x \in \{0,1\}^m, \bm y \in \{0,1\}^n
				\end{subarray} } \bigg\{ \sum_{i\in [m]} \sum_{j\in [n]} A_{ij}^2 x_{i} y_j: \sum_{i\in [m]} x_i \le s_1,  \sum_{j\in [n]} y_j \le s_2 \bigg\}.
			\end{align*}
			Above, introducing binary variables $	\bm z\in \{0,1\}^{m \times n}$ and using McCormick inequalities \cite{mccormick1976computability} to linearize the bilinear terms in the objective function, we arrive at the desirable MILP formulation.
		\end{enumerate}
	\end{subequations}
	\qed
\end{proof}

\subsection{Proof of Theorem~\ref{thm_bc}} \label{proof:thm_bc}
\thmbc*
\begin{proof}	
	\begin{subequations}
		By separating the binary variables $\bm z$, we can rewrite the SPCA \eqref{eq_spcasdp2} as
		\begin{align} \label{eq_spca1}
		\max_{\begin{subarray}{c}
			\bm z \in \{0,1\}^n, \\
			\sum_{i\in [n]} z_i \le s
			\end{subarray} } f(\bm{z}):=\max_{ \begin{subarray}{c}\bm X  \in \S_+^r, \bm X \preceq \bm I_r, \\
			\bm W_1, \cdots, \bm W_n \in \S_+^r\end{subarray}} \bigg\{\sum_{i\in [n]}\bm c_i^{\top} \bm W_i \bm c_i:  \tr(\bm X) \le k, \bm X \succeq \bm W_i, \tr(\bm W_i) \le z_i, \forall i\in [n]     \bigg\}.
		\end{align}
		
		For the inner maximization problem above, we introduce Lagrangian multipliers $\{\bm Q_i \in \S_+^r \}_{i\in [n]}$ 
		and $\bm \mu \in \Re_+^n$ for the constraints with respect to matrix variables $\{ \bm W_i\}_{i\in [n]}$ and the resulting dual problem is 
		\begin{align*}
		f(\bm{z})\leq \min_{\begin{subarray}{c}
			\bm \mu \in \Re_+^n,\\
			\bm Q_1, \cdots, \bm Q_n \in \S_+^r
			\end{subarray}} 
		\max_{ \begin{subarray}{c}
			\bm W_1, \cdots, \bm W_n \in \S_+^r,\\
			\bm X \in \S_+^r, \bm X \preceq \bm I_r, \tr(\bm X) \le k
			\end{subarray}} \bigg\{\sum_{i\in [n]}\bm c_i^{\top} \bm W_i \bm c_i  + \sum_{i\in [n]}\tr(\bm Q_i(\bm X - \bm W_i)) + \sum_{ i \in [n]} \mu_i(z_i - \tr(\bm W_i))   \bigg\},
		\end{align*}
		where the inequality is by weak duality.
		Maximizing the inner problem above over variables $\{ \bm W_i\}_{i\in [n]}$ yields 
		\begin{align*}
		f(\bm{z})\leq 	\min_{\begin{subarray}{c}
			\bm \mu \in \Re_+^n,\\
			\bm Q_1, \cdots, \bm Q_n \in \S_+^r
			\end{subarray}} 
		\max_{ \begin{subarray}{c}
			\bm X \in \S_+^r, \bm X \preceq \bm I_r, \tr(\bm X) \le k
			\end{subarray}} \bigg\{ \sum_{ i \in [n]}\tr(\bm Q_i\bm X) + \sum_{ i \in [n]} \mu_i z_i : \bm c_i \bm c_i^{\top} \preceq \bm Q_i + \mu_i \bm I_r, \forall i\in [n]   \bigg\}.
		\end{align*}
		Next, from Lemma~\ref{lem:eigen}, maximizing the inner problem above over the matrix variable $\bm X$ yields 
		\begin{align} \label{eq_dual}
		f(\bm{z})\leq \min_{\begin{subarray}{c}
			\bm \mu \in \Re_+^n,
			\bm Q_1, \cdots, \bm Q_n \in \S_+^r
			\end{subarray}} 
		\bigg\{ \bigg|\bigg|\sum_{ i \in [n]} \bm Q_i\bigg|\bigg|_{(k)} + \sum_{ i \in [n]} \mu_iz_i : \bm c_i \bm c_i^{\top} \preceq \bm Q_i + \mu_i \bm I_r, \forall i\in [n]   \bigg\}.
		\end{align}
		According to the proof of Corollary~\ref{cor2}, for any feasible solution $\bm z$ of SPCA \eqref{eq_spca1} with its support $\hat{S}$, the inner maximization problem in \eqref{eq_spca1} admits the optimal value $f(\bm{z})=||\sum_{i\in \hat{S}}\bm c_i \bm c_i^{\top}||_{(k)} $, implying that the optimal value of its corresponding dual problem \eqref{eq_dual} must be lower bounded by $||\sum_{i\in \hat{S}}\bm c_i \bm c_i^{\top}||_{(k)}$ based on the weak duality.
		
		In particular, for a subset $S\subseteq [n]$ of size at most $s$, we can construct a feasible solution to the dual  problem \eqref{eq_dual} as  
		\begin{align*}
		\bm Q_i = \bm c_i \bm c_i^{\top}, \mu_i = 0, \forall i\in {S}, \bm Q_i = \bm 0_{r,r}, \mu_i = \bm c_i^{\top}\bm c_i, \forall i\in [n]\setminus S.
		\end{align*}
		Plugging the above solution into problem \eqref{eq_dual} and finding the minimum of objective values led by them,  problem \eqref{eq_dual} can be upper bounded by
		\begin{align} \label{eq_subprob}
		f(\bm{z})\leq 	\min_{S\subseteq [n], |S|\le s} \bigg\{ \bigg|\bigg|\sum_{i\in {S}}\bm c_i \bm c_i^{\top}\bigg|\bigg|_{(k)} + \sum_{i\in [n]\setminus S}|| \bm c_i||_2^2  z_i\bigg\}\leq \bigg|\bigg|\sum_{i\in \hat{S}}\bm c_i \bm c_i^{\top}\bigg|\bigg|_{(k)}: =f(\bm{z}),
		\end{align}
		where the equality can be achieved by the support $\hat{S}$ of $\bm z$ since
		\begin{align*} 
		\bigg|\bigg|\sum_{i\in \hat{S}}\bm c_i \bm c_i^{\top}\bigg|\bigg|_{(k)}  +\sum_{i\in [n]\setminus \hat{S}}|| \bm c_i||_2^2  z_i = \bigg|\bigg|\sum_{i\in \hat{S}}\bm c_i \bm c_i^{\top}\bigg|\bigg|_{(k)}.
		\end{align*}
		
		Therefore, SPCA \eqref{eq_spca1} can be formulated by
		\begin{align} \label{eq_spca2}
		\max_{\begin{subarray}{c}
			\bm z \in \{0,1\}^n, \\
			\sum_{i\in [n]} z_i \le s
			\end{subarray} } 
		\min_{S\subseteq [n], |S|\le s} \bigg\{ \bigg|\bigg|\sum_{i\in {S}}\bm c_i \bm c_i^{\top}\bigg|\bigg|_{(k)} + \sum_{i\in [n]\setminus S}|| \bm c_i||_2^2  {z}_i\bigg\} .
		\end{align}
		Introducing the variable $w$ to linearize the inner minimization above and plugging the identities that $||\bm A_{S, S}||_{(k)} = ||\sum_{i\in {S}}\bm c_i \bm c_i^{\top}||_{(k)}$ for any subset $S\subseteq [n]$ and $A_{ii} = \bm c_i^{\top} \bm c_i$ for any $i\in [n]$, we obtain the formulation \eqref{eq_spcabc}.
	\end{subequations} 
	\qed
\end{proof}

\subsection{Proof of Proposition~\ref{prop3}} \label{proof:prop3}
\prop*
\begin{proof}
	We show the equivalence of SSVD \eqref{eq_ssvdcom} and problem \eqref{eq_ssvdcom2} via one-to-one solution correspondence. First, for any feasible solution $(S_1, S_2)$ to SSVD, let us construct set $S\subseteq [m+n]$ as
	\begin{align*}
	S := \{i: i\in S_1\}\cup\{j+m: j\in S_2\}.
	\end{align*} 
	Then we have
	\begin{equation}\label{eq_ssvd_symm}
	\begin{aligned}
	||\overline{\bm A}_{S, S}||_{(k)} &= \bigg|\bigg| \begin{pmatrix}
	\bm 0 & \bm A_{S_1, S_2}\\
	\bm A_{S_1, S_2}^{\top} & \bm 0 
	\end{pmatrix} + \sigma_{1}(\bm A) (\bm I_{m+n})_{S, S} \bigg|\bigg|_{(k)} = \sum_{i\in [k]}(\sigma_{i}(\bm A_{S_1, S_2}) + \sigma_{1}(\bm A) ) \\
	&= ||\bm A_{S_1,S_2}||_{(k)} + k \sigma_{1}(\bm A),
	\end{aligned}
	\end{equation}
	where the first equality is from the definition of matrix $\overline{\bm A}$ and the second equality is due to Lemma~\ref{lem:aug}.
	
	Second, for any feasible subset $S$ to problem  \eqref{eq_ssvdcom2}, we construct two subsets $S_1\subseteq [m]$ and $S_2\subseteq [n]$  as 
	\begin{align*}
	S_1 =\{i: i\in[m]\cap S\}, S_2 =\{j-m: j\in[m+1, m+n]\cap S\},
	\end{align*}
	which is feasible to SSVD \eqref{eq_ssvdcom}. Similarly, according to \eqref{eq_ssvd_symm}, we have $||\bm A_{S_1,S_2}||_{(k)} = ||\overline{\bm A}_{S, S}||_{(k)}-  k \sigma_{1}(\bm A)$.
	
	This proves the equivalence of SSVD \eqref{eq_ssvdcom} and problem \eqref{eq_ssvdcom2}. \qed
\end{proof}

\end{appendices}

\end{document}